\documentclass[12pt,british]{amsart}
\usepackage{babel}
\usepackage[T1]{fontenc}
\usepackage[utf8]{inputenc}
\usepackage{geometry}
\usepackage{units}
\usepackage{mathtools}
\usepackage{amsmath}
\usepackage{amsthm}
\usepackage{amssymb}
\usepackage{stmaryrd}
\usepackage{graphicx}
\usepackage{enumitem}
\usepackage[unicode=true,pdfusetitle,
 bookmarks=true,bookmarksnumbered=false,bookmarksopen=false,
 breaklinks=false,pdfborder={0 0 1},backref=false,colorlinks=false]
 {hyperref}

\makeatletter

\theoremstyle{plain}
\newtheorem{thm}{Theorem}
\newtheorem{cor}[thm]{Corollary}

\newtheorem{lem}[thm]{Lemma}
\newtheorem{prop}[thm]{Proposition}

\theoremstyle{definition}
\newtheorem{defn}[thm]{Definition}

\theoremstyle{remark}
\newtheorem{rem}[thm]{Remark}
\newtheorem{notation}[thm]{Notation}

\newenvironment{subproof}[1][\proofname]{%
    \begin{proof}[#1]%
    }{\end{proof}}

\newlist{casenv}{enumerate}{4}
\setlist[casenv]{leftmargin=*,align=left,widest={iiii}}
\setlist[casenv,1]{label={{\itshape\ Case} \arabic*.},ref=\arabic*}
\setlist[casenv,2]{label={{\itshape\ Case} \roman*.},ref=\roman*}
\setlist[casenv,3]{label={{\itshape\ Case \alph*.}},ref=\alph*}
\setlist[casenv,4]{label={{\itshape\ Case} \arabic*.},ref=\arabic*}

\makeatother

\DeclarePairedDelimiter{\paren}{\lparen}{\rparen}%
\DeclarePairedDelimiter{\pbrace}{\lbrace}{\rbrace}%
\DeclarePairedDelimiter{\abs}{\lvert}{\rvert}%
\DeclarePairedDelimiter{\norm}{\lVert}{\rVert}%

\newcommand{\expfrac}[2]{{\scriptstyle #1}/\raisebox{-1pt}{\ensuremath{{\scriptstyle #2}}}}

\newcommand{\Aut}{\operatorname{Aut}}
\newcommand{\diag}{\operatorname{diag}}
\newcommand{\id}{\operatorname{id}}
\newcommand{\Jac}{\operatorname{Jac}}
\renewcommand{\Re}{\operatorname{Re}}

\begin{document}

\title[Attracting Fatou cycles of type $\mathbb{C}\times(\mathbb{C}^{*})^{d-1}$]{Periodic cycles of attracting Fatou components of type $\mathbb{C}\times(\mathbb{C}^{*})^{d-1}$
in automorphisms of $\mathbb{C}^{d}$}
\author[J. Reppekus]{Josias Reppekus}
\thanks{The author acknowledges the MIUR Excellence Department Project awarded
to the Department of Mathematics, University of Rome Tor Vergata,
CUP E83C18000100006}
\address{Dipartimento di Matematica\\
Universit\`{a} di Roma ``Tor Vergata''\\
Via Della Ricerca Scientifica 1 \\
00133, Roma, Italy}
\email{reppekus@mat.uniroma2.it}


\begin{abstract}
We generalise a recent example by F.~Bracci, J.~Raissy and B.~Stens{\o}nes
to construct automorphisms of $\mathbb{C}^{d}$ admitting an arbitrary
finite number of non-recurrent Fatou components, each biholomorphic
to $\mathbb{C}\times(\mathbb{C}^{*})^{d-1}$ and all attracting to
a common boundary fixed point. These automorphisms can be chosen such
that each Fatou component is invariant or such that the components
are grouped into periodic cycles of any common period. We further
show that no orbit in these attracting Fatou components can converge
tangent to a complex submanifold, and that every stable orbit near
the fixed point is contained either in these attracting components
or in one of $d$ invariant hypersurfaces tangent to each coordinate
hyperplane on which the automorphism acts as an irrational rotation.
\end{abstract}

\maketitle

\tableofcontents{}
\newpage{}

\section*{Introduction}

When studying the behaviour of iterates of a holomorphic endomorphism
$F$ of $\mathbb{C}^{d}$, $d\ge1$, one of the basic objects of interest
is the \emph{Fatou set}
\[
\mathcal{F}:=\pbrace{z\in\mathbb{C}^{d}\mid{{\pbrace{F^{n}}}}_{n\in\mathbb{N}}\text{ is normal on a neighbourhood of }z}.
\]
A connected component of $\mathcal{F}$ is called a \emph{Fatou component}
of $F$. Let $V$ be a Fatou component of $F$. Then $V$ is \emph{invariant},
if $F(V)=V$. More generally, $V$ is \emph{$p$-periodic} for $p\in\mathbb{N}^{*}$,
if $F^{p}(V)=V$ and $F^{q}(V)\neq V$ for $q<p$. In this case we
call $(V,F(V),\cdots,F^{p-1}(V))$ a \emph{$p$-periodic cycle} of
Fatou components. A Fatou component $V$ is \emph{attracting to $P\in\mathbb{C}^{d}$},
if $(F|_{V})^{n}\to P$ (then in particular $F(P)=P$). A periodic
Fatou component $V$ attracting to $P$ is \emph{recurrent} if $P\in V$
and \emph{non-recurrent} if $P\in\partial V$.

Every recurrent attracting Fatou component of an automorphism of $\mathbb{C}^{d}$
is biholomorphic to $\mathbb{C}^{d}$ (this follows from \cite[Theorem~2]{PetersVivasWold2008AttractingbasinsofvolumepreservingautomorphismsofmathbbCk}
and the appendix of \cite{RosayRudin1988HolomorphicmapsfrombfCntobfCn}).
For \emph{polynomial} automorphisms of $\mathbb{C}^{2}$, even non-recurrent
attracting periodic Fatou components are biholomorphic to $\mathbb{C}^{2}$
(by \cite[Theorem~6]{LyubichPeters2014ClassificationofinvariantFatoucomponentsfordissipativeHenonmaps}
and \cite{Ueda1986LocalstructureofanalytictransformationsoftwocomplexvariablesI}).

In \cite{BracciRaissyStensonesAutomorphismsofmathbbCkwithaninvariantnonrecurrentattractingFatoucomponentbiholomorphictomathbbCtimesmathbbCk1},
F.\ Bracci, J.\ Raissy and B.\ Stens{\o}nes proved the existence
of automorphisms of $\mathbb{C}^{d}$ with a non-recurrent attracting
invariant Fatou component biholomorphic to $\mathbb{C}\times(\mathbb{C}^{*})^{d-1}$.
In particular this provided first examples of automorphisms of $\mathbb{C}^{2}$
with a multiply connected attracting Fatou component (those are necessarily
non-polynomial by the previously mentioned results). Based on this,
it is easy to construct automorphisms of $\mathbb{C}^{d}$ with non-recurrent
attracting invariant Fatou components biholomorphic to $\mathbb{C}^{d-m}\times(\mathbb{C}^{*})^{m}$
for $m<d$ (see Corollary~\ref{cor:CCstarVariants}).

By \cite[Proposition~5.1]{Ueda1986LocalstructureofanalytictransformationsoftwocomplexvariablesI},
attracting Fatou components are Runge, and, by \cite{Serre1955UneproprietetopologiquedesdomainesdeRunge},
for every Runge domain $D\subseteq\mathbb{C}^{d}$, we have $H^{q}(D)=0$
for $q\ge d$. Hence $\mathbb{C}\times(\mathbb{C}^{*})^{d-1}$ has
the highest possible degree of non-vanishing cohomology for an attracting
Fatou component. It is an open question whether all non-recurrent
attracting invariant Fatou components of automorphisms are biholomorphic
to a product of copies of $\mathbb{C}$ and $\mathbb{C}^{*}$. To
the author's knowledge it is not even clear these are the only homotopy
types that can occur.

Non-recurrent attracting Fatou components of type $\mathbb{C}^{d}$
appear in parabolic flowers (generalisations of one-dimensional Leau-Fatou
flowers), that is in arbitrary finite number around a fixed point
and grouped in periodic cycles. In this paper we generalise the example
of \cite{BracciRaissyStensonesAutomorphismsofmathbbCkwithaninvariantnonrecurrentattractingFatoucomponentbiholomorphictomathbbCtimesmathbbCk1}
to higher orders to show that the same can occur for type $\mathbb{C}\times(\mathbb{C}^{*})^{d-1}$.
We further extend their results to provide a complete classification
of stable orbits near the fixed point in these examples.

We will be studying germs $F$ of automorphisms of $\mathbb{C}^{d}$
at the origin of the form
\begin{equation}
F(z^{1},\ldots,z^{d})=(\lambda_{1}z^{1},\ldots,\lambda_{d}z^{d})\paren [\Big]{1-\frac{(z^{1}\cdots z^{d})^{k}}{kd}}+O(\norm z^{l}),\label{eq:FIntro}
\end{equation}
where $\lambda_{1},\ldots,\lambda_{d}$ are of unit modulus, not roots
of unity, such that $F$ is one-resonant via $\lambda_{1}\cdots\lambda_{d}=1$,
i.e.\ $\lambda_{1}^{m_{1}}\cdots\lambda_{d}^{m_{d}}=\lambda_{j}$
for $m_{1},\ldots,m_{d}\in\mathbb{N}$ and $j\in\{1,\ldots,d\}$ if
and only if $(m_{1},\ldots,m_{d})=(q,\ldots,q)+e_{j}$ for some $q\in\mathbb{N}$
(see Definition~\ref{def:oneRes}), and $l>2kd+1$. In some parts
we will in addition assume all subsets $\{\lambda_{1},\ldots,\lambda_{d}\}\backslash\{\lambda_{j}\}$,
$j=1,\ldots,d$ to satisfy the Brjuno condition (Definition~\ref{def:BrjunoAndPart}).
For $k=1$ this is precisely the set-up of \cite{BracciRaissyStensonesAutomorphismsofmathbbCkwithaninvariantnonrecurrentattractingFatoucomponentbiholomorphictomathbbCtimesmathbbCk1}.

Our main results are the following:
\begin{thm}
\label{thm:FatouCCstarMulti}Let $F$ be a germ of automorphisms of
$\mathbb{C}^{d}$ at the origin of the form (\ref{eq:FIntro}). Then
$F$ admits $k$ disjoint, completely invariant \emph{(}$F(\Omega_{h})=\Omega_{h}=F^{-1}(\Omega_{h})$\emph{)},
attracting basins $\Omega_{0},\ldots,\Omega_{k-1}$ such that
\begin{enumerate}
\item If each subset $\{\lambda_{1},\ldots,\lambda_{d}\}\backslash\{\lambda_{j}\}$,
$j=1,\ldots,d$ satisfies the Brjuno condition, then:
\begin{enumerate}
\item $\Omega_{h}$ is a union of Fatou components for each $h=0,\ldots,k-1$,
\item $F$ admits Siegel hypersurfaces (i.e.~invariant hypersurfaces on
which $F$ acts as a rotation) tangent to each coordinate hyperplane,
\item All stable orbits of $F$ near the origin are contained in one of
the above.
\end{enumerate}
\item If $F$ is a global automorphisms of $\mathbb{C}^{d}$, then for each
$h=0,\ldots,k-1$ there exists a biholomorphic map $\phi_{h}:\Omega_{h}\to\mathbb{C}\times(\mathbb{C}^{*})^{d-1}$
conjugating $F$ to
\[
(\zeta^{1},\cdots,\zeta^{d})\mapsto(\zeta^{1}+1,\zeta^{2},\ldots,\zeta^{d}).
\]
Moreover, there exist automorphisms of the form (\ref{eq:FIntro})
for each admissible choice of $\lambda_{1},\ldots,\lambda_{d}$ and
$l>2kd+1$.
\end{enumerate}
\end{thm}

\begin{rem}
\label{rem:geometryIntro}Each global basin $\Omega_{h}$ arises as
the union of all iterated preimages of an explicit local attracting
basin $B_{h}$ of the desired homotopy type whose external geometry
becomes apparent in polar decomposition as depicted in Figure~\ref{fig:Argument-components}
for $d=k=2$. The global basins are more abstract, so we don't know
much about their outer shape or arrangement.
\end{rem}

\begin{rem}
\label{rem:OrbitsIntro}Each attracting orbit in a basin $\Omega_{h}$
converges tangent to a real $d$-dimensional submanifold (depending
on the orbit), but not tangent to any complex subspace.
\end{rem}

\begin{thm}
\label{thm:FatouCCstarPeriodic}Let $p\in\mathbb{N}^{*}$ divide $k$.
Then there exist automorphisms $G$ of $\mathbb{C}^{d}$ such that
$G^{p}$ has the form (\ref{eq:FIntro}) and $G(\Omega_{h})=\Omega_{h+p\bmod k}$
for $h=0,\ldots,k-1$. In particular, $G$ admits $k/p$ disjoint
$p$-cycles of non-recurrent, attracting Fatou components biholomorphic
to $\mathbb{C}\times(\mathbb{C}^{*})^{d-1}$, that are all attracted
to the origin.
\end{thm}

As an immediate corollary, we obtain automorphisms with cycles of
non-recurrent attracting Fatou components biholomorphic to any product
of copies of $\mathbb{C}$ and $\mathbb{C}^{*}$ with admissible cohomology:
\begin{cor}
\label{cor:CCstarVariants}Let $d,k\in\mathbb{N}^{*}$, $p\in\mathbb{N}^{*}$
divide $k$, and $0\le m<d$. Then there exist holomorphic automorphisms
of $\mathbb{C}^{d}$ possessing $k/p$ disjoint $p$-cycles of non-recurrent,
attracting, invariant Fatou components biholomorphic to $\mathbb{C}^{d-m}\times(\mathbb{C}^{*})^{m}$
and attracted to the origin.
\end{cor}

\begin{figure}
$\begin{array}{c}\includegraphics[width=0.4\columnwidth]{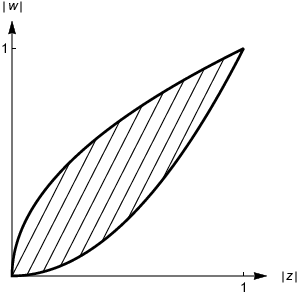}\end{array}\bigtimes\quad{}\begin{array}{c}\includegraphics[bb=0bp 0bp 263bp 191bp,width=0.5\columnwidth]{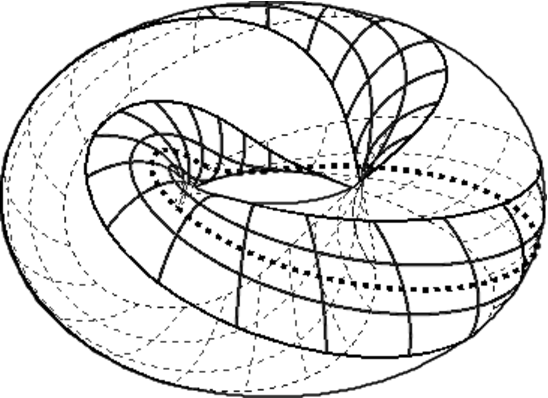}\end{array}$

\caption{\label{fig:Argument-components}Decomposition in modulus and argument
components for two local basins $B_{0}$ and $B_{1}$ with central
curve $\arg z+\arg w\equiv0$ in $B_{0}$}
\end{figure}

We also prove an auxiliary result on holomorphic elimination of infinite
families of monomials that may be interesting in its own right. We
use multi index notation $\lambda^{\alpha}=\lambda_{1}^{\alpha_{1}}\cdots\lambda_{d}^{\alpha_{d}}$
for $\alpha=(\alpha_{1},\ldots,\alpha_{d})\in\mathbb{N}^{d}$ and
define the notion of a \emph{Brjuno set} of exponents $A\subseteq\mathbb{N}^{d}$,
by requiring a Brjuno condition only on the small divisors $\lambda^{\alpha}-\lambda_{j}$
with $\alpha\in A$ and $j\in\{1,\ldots,d\}$ (see Definition~\ref{def:BrjunoSet}).
\begin{thm}
\label{thm:IteratedEliminationSimple}Let $F$ be a germ of endomorphisms
of $\mathbb{C}^{d}$ of the form $F(z)=\Lambda z+\sum_{|\alpha|>1}\sum_{j=1}^{d}f_{\alpha}^{j}z^{\alpha}e_{j}$
with $\Lambda=\diag(\lambda_{1},\ldots,\lambda_{d})$. Let $A_{0}$
and $A$ be disjoint sets of multi-indices in $\mathbb{N}^{d}$ such
that $A$ admits a partition $A=A_{1}\cup\cdots\cup A_{k_{0}}$ such
that
\begin{enumerate}
\item \label{enu:itElim1}For $0\le k\le k_{0}$, if $\alpha\in A_{k}$
and $\beta\le\alpha$, then $\beta\in A_{\overline{k}}:=A_{0}\cup\cdots\cup A_{k}$
(where $\le$ is taken component-wise).
\item \label{enu:itElim2}For $1\le k\le k_{0}$, if $\beta_{1},\ldots,\beta_{l}\in A_{0}$
such that $\beta_{1}+\cdots+\beta_{l}\in A_{\overline{k}}$, $|\beta_{1}|\ge2$,
and $f_{\beta_{1}}^{j_{1}}\cdots f_{\beta_{l}}^{j_{l}}\neq0$, then
$e_{j_{1}}+\cdots+e_{j_{l}}\notin A_{k}$.
\item \label{enu:itElim3}$A$ is a Brjuno set for $F$.
\end{enumerate}
Then there exists a local biholomorphism $H\in\Aut(\mathbb{C}^{d},0)$
conjugating $F$ to $G=H^{-1}\circ F\circ H$ where $G(z)=\sum_{|\alpha|>1}g_{\alpha}z^{\alpha}$
with $g_{\alpha}=f_{\alpha}$ for $\alpha\in A_{0}$ and $g_{\alpha}=0$
for $\alpha\in A$.
\end{thm}

The proof of the theorem is based on that of a partial linearisation
result from \cite{Poeschel1986Oninvariantmanifoldsofcomplexanalyticmappingsnearfixedpoints}
which it generalises.

\subsection*{Outline}

In Section~\ref{sec:LocalConstruction}, following \cite{BracciRaissyStensonesAutomorphismsofmathbbCkwithaninvariantnonrecurrentattractingFatoucomponentbiholomorphictomathbbCtimesmathbbCk1},
we recall results from \cite{BracciZaitsev2013Dynamicsofoneresonantbiholomorphisms}
that show that germs of the form (\ref{eq:FIntro}) have $k$ local
attracting basins of the desired homotopy type. We then examine their
arrangement in the surrounding space.

In Section~\ref{sec:Elimination} we prove Theorem~\ref{thm:IteratedEliminationSimple}
and, under the aforementioned Brjuno-type condition, we conclude the
existence of local coordinates that allow us to better control the
unknown tail of $F$.

In Section~\ref{sec:Stable-Orbits} we use those coordinates to extend
\cite[Lemma~2.5]{BracciRaissyStensonesAutomorphismsofmathbbCkwithaninvariantnonrecurrentattractingFatoucomponentbiholomorphictomathbbCtimesmathbbCk1}
to classify the stable orbits of $F$ near the origin, proving the
first part of Theorem~\ref{thm:FatouCCstarMulti}.

In Section~\ref{sec:Internal-geometry} we define two closely related
systems of coordinates on each local basin compatible with the action
of $F$: the first, in a small variation of \cite[Section~3]{BracciRaissyStensonesAutomorphismsofmathbbCkwithaninvariantnonrecurrentattractingFatoucomponentbiholomorphictomathbbCtimesmathbbCk1},
allows us to study the behaviour of attracting orbits more carefully
in Section~\ref{subsec:InternalOrbitBehaviour}, showing Remark~\ref{rem:OrbitsIntro};
the second in Section~\ref{subsec:Global-Basins} conjugates $F$
to an affine map and, if $F$ is an automorphism, extends to a biholomorphism
from the corresponding global basin to $\mathbb{C}\times(\mathbb{C}^{*})^{d-1}$.
The existence of automorphisms of the form (\ref{eq:FIntro}) follows
directly from a jet-interpolation result, concluding the proof of
the second part of Theorem~\ref{thm:FatouCCstarMulti}.

Finally, we show Theorem~\ref{thm:FatouCCstarPeriodic} and Corollary~\ref{cor:CCstarVariants}
in Section~\ref{sec:Periodic-cycles} via explicit calculations.

\subsection*{Conventions}

$\Aut(\mathbb{C}^{d})$ is the set of biholomorphic automorphisms
of $\mathbb{C}^{d}$ and $\Aut(\mathbb{C}^{d},0)$ the set of germs
of biholomorphisms of $\mathbb{C}^{d}$ at the origin such that $F(0)=0$.

For $z\in\mathbb{C}^{d}$ and $F\in\Aut(\mathbb{C}^{d},0)$, upper
indices denote the components of $z=(z^{1},\ldots,z^{d})$, while
a lower index $n\in\mathbb{N}:=\{0,1,\ldots\}$ denotes the iterated
image $z_{n}=(z_{n}^{1},\ldots,z_{n}^{d}):=F^{n}(z)$ of $z$ under
$F$. Similarly, for the coordinates $u:=\pi(z):=z^{1}\cdots z^{d}$,
$U:=u^{-k}$ we set $u_{n}:=\pi(z_{n})$ and $U_{n}:=u_{n}^{-k}$.

For $\{x_{n}\}_{n\in\mathbb{N}}$ a sequence of objects, we say the
object $x_{n}$ has a property \emph{eventually}, if there exists
$n_{0}\in\mathbb{N}$ such that $x_{n}$ has this property for all
$n>n_{0}$.

For a topological space $D$ and maps $f,g:D\to\mathbb{C}$ , we use
Bachmann-Landau notation for global behaviour:
\begin{itemize}
\item $f(x)=O(g(x))$ for $x\in D$, if $|f(x)|\le C|g(x)|$ for all $x\in D$
for some $C>0$,
\item $f(x)\approx g(x)$ for $x\in D$, if $f(x)=O(g(x))$ and $g(x)=O(f(x))$
(often denoted $f(x)=\Theta(g(x))$),
\end{itemize}
and for asymptotic behaviour:
\begin{itemize}
\item $f(x)=O(g(x))$ as $x\to x_{0}$, if $\limsup_{x\to x_{0}}\frac{|f(x)|}{|g(x)|}=C<+\infty$,
\item $f(x)\approx g(x)$ as $x\to x_{0}$, if $f(x)=O(g(x))$ and $g(x)=O(f(x))$
as $x\to x_{0}$,
\item $f(x)=o(g(x))$ as $x\to x_{0}$, if $\lim_{x\to x_{0}}\frac{|f(x)|}{|g(x)|}=0$,
\item $f(x)\sim g(x)$ as $x\to x_{0}$, if $\lim_{x\to x_{0}}\frac{f(x)}{g(x)}=1$
or $f(x)=g(x)(1+o(1))$ as $x\to x_{0}$.
\end{itemize}

\numberwithin{thm}{section}
\numberwithin{equation}{section}

\section{\label{sec:LocalConstruction}Local basins of attraction}

After recalling a construction of local basins of attraction, we give
their representation in internal holomorphic coordinates to determine
their homotopy type, and in external polar coordinates to visualise
their arrangement in $\mathbb{C}^{d}$.

The local basins arise from the study of local dynamics of one-resonant
germs in \cite{BracciZaitsev2013Dynamicsofoneresonantbiholomorphisms}
by F.~Bracci and D.~Zaitsev.
\begin{defn}
\label{def:oneRes}A germ $F$ of endomorphisms of $\mathbb{C}^{d}$
at the origin such that $F(0)=0$ and $dF_{0}=\diag(\lambda_{1},\ldots,\lambda_{d})$
is called \emph{one-resonant} of index $\alpha=(\alpha_{1},\ldots,\alpha_{d})\in\mathbb{N}^{d}$,
if $\lambda_{j}=\lambda_{1}^{m_{1}}\cdots\lambda_{d}^{m_{d}}$ for
some $j\le d$ and $m=(m_{1},\ldots,m_{d})\in\mathbb{N}^{d}$ if and
only if $m=k\alpha+e_{j}$ for some $k\in\mathbb{N}$ (where $e_{j}$
denotes the $j$-th unit vector).
\end{defn}

\begin{rem}
For $1\le j\le d$, one-resonance of index $\alpha\neq n\cdot e_{j}$
for every $n\in\mathbb{N}$ implies in particular that $\lambda_{j}$
is not a root of unity.
\end{rem}

We start with germs of biholomorphisms of $\mathbb{C}^{d}$ at the
origin in normal form $F_{{\rm N}}$ given for $z=(z^{1},\ldots,z^{d})$
by
\begin{equation}
F_{\mathrm{N}}\paren z=\Lambda z\cdot\paren [\Big]{1-\frac{(z^{1}\cdots z^{d})^{k}}{kd}},\label{eq:FNormalForm}
\end{equation}
that are one-resonant of index $(1,\ldots,1)$ with $\Lambda=\diag(\lambda_{1},\ldots,\lambda_{d})$
such that $|\lambda_{j}|=1$ for each $j\le d$. We will later moreover
assume that proper subsets of $\{\lambda_{1},\ldots,\lambda_{d}\}$
satisfy the Brjuno condition (see Definition~\ref{def:BrjunoAndPart}).

An important tool to study the dynamics of this type of maps introduced
in \cite{BracciZaitsev2013Dynamicsofoneresonantbiholomorphisms} is
the variable $u=\pi(z)=z^{1}\cdots z^{d}$ on which $F_{{\rm N}}$
acts parabolically of order $k$ near the origin as $u\mapsto u(1-u^{k})+O(u^{2k+1})$,
yielding a Leau-Fatou flower of $k$ attracting sectors
\[
S_{h}(R,\theta):=\pbrace [\big]{u\in\mathbb{C}\mid{\abs [\big]{u^{k}-\tfrac{1}{2R}}}<\tfrac{1}{2R},{\abs [\big]{\arg(u)-\tfrac{2\pi h}{k}}}<\theta}
\]
for $h=0,\ldots,k-1$ and suitable $R>0$ and $\theta\in(0,\nicefrac{\pi}{2k})$.
Note that on each such sector the map $u\mapsto u^{-k}=:U$ is injective,
hence each sector is biholomorphic to a ``sector at infinity''
\[
H(R,\theta):=\{U\in\mathbb{C}\mid\Re U>R,\abs{\arg(U)}<k\theta\}.
\]
To control $z$ in terms of $u=\pi(z)$, for $\beta\in(0,1/d)$ let
further
\[
W(\beta):=\{z\in\mathbb{C}^{d}\mid|z^{j}|<|\pi(z)|^{\beta}\text{ for }j\le d\},
\]
and for $h=0,\ldots,k-1$
\[
B_{h}(R,\theta,\beta):=\{z\in W(\beta)\mid\pi(z)\in S_{h}(R,\theta)\}.
\]
Now from the proof of \cite[Theorem~1.1]{BracciZaitsev2013Dynamicsofoneresonantbiholomorphisms}
it follows:
\begin{thm}
\label{thm:BZ}Let $F_{{\rm N}}$ be of the form (\ref{eq:FNormalForm})
and $l\in\mathbb{N}$, $l>2kd+1$. Then for every germ $F$ of automorphisms
of $\mathbb{C}^{d}$ at the origin of the form
\begin{equation}
F(z)=F_{{\rm N}}(z)+O(\norm z^{l}),\label{eq:BZgerms}
\end{equation}
for every $\beta_{0}\in(0,1/d)$ such that $\beta_{0}(l+d-1)>2k+1$,
and every $\theta_{0}\in(0,\nicefrac{\pi}{2k})$, there exists $R_{0}>0$
such that the (disjoint, non-empty) open sets $B_{h}:=B_{h}(R_{0},\theta_{0},\beta_{0})$
for $h=0,\ldots,k-1$ are uniform local basins of attraction for $F$,
that is $F(B_{h})\subseteq B_{h}$, and $\lim_{n\to\infty}F^{n}\equiv0$
uniformly in $B_{h}$ for each $h$.
\end{thm}

\begin{rem}
\label{rem:HomotopyType}As in \cite[Lemma~2.7 and Section~7]{BracciRaissyStensonesAutomorphismsofmathbbCkwithaninvariantnonrecurrentattractingFatoucomponentbiholomorphictomathbbCtimesmathbbCk1},
we observe that each local basin $B_{h}$ is homotopy equivalent to
$(S^{1})^{d-1}$, so the local basins have the desired homotopy type
(of $\mathbb{C}\times(\mathbb{C}^{*})^{d-1}$).

To see this, let again $u=z^{1}\cdots z^{d}$. Then $(u,z')=(u,z^{2},\ldots,z^{d})$
is a holomorphic system of coordinates on $B_{h}$ through which $B_{h}$
is biholomorphic to
\begin{multline*}
\{(u,z')\in(\mathbb{C}^{*})^{d}\mid u\in S_{h}(R_{0},\theta_{0}),|u|^{1-\beta_{0}}<|z^{2}\cdots z^{d}|,|z^{j}|<|u|^{\beta_{0}}\text{ for }j\ge2\}\\
=\{u\in S_{h}(R_{0},\theta_{0}),|u|^{1-(d-j+1)\beta_{0}}|z^{2}\cdots z^{j-1}|^{-1}<|z^{j}|<|u|^{\beta_{0}}\text{ for }j\ge2\}.
\end{multline*}
Since the sector $S_{h}(R_{0},\theta_{0})$ is contractible, and
for each $j$ given $u,z^{2},\ldots,z^{j}$, the value of $z^{j+1}$
is confined to an annulus, this is homotopy equivalent to $(S^{1})^{d-1}$.
\end{rem}

\begin{rem}[Shape and arrangement]
\label{rem:externalGeom}The external shape of the local basins becomes
apparent in polar coordinates. For the sake of visualisation and simplicity,
let $d=2$ and assume $0<R<1$. To get a global (real) smooth argument
coordinate, we consider $\arg$ to take values in $S^{1}=\mathbb{R}/(2\pi\mathbb{Z})$,
so $(\arg(z),\arg(w))$ is a point on the torus $\mathbb{T}^{2}=S^{1}\times S^{1}$.
For $0<R<1$ the condition ${\abs [\big]{u^{k}-\frac{1}{2R}}}<\frac{1}{2R}$
is implied by the others and so for $h\in\{0,\ldots,k-1\}$, we have
\[
B_{h}=\pbrace{(z,w)\in W(\beta)\mid{{\abs{\arg(z)+\arg(w)-2\pi h/k}}}<\theta}.
\]
In this case $B_{h}$ is diffeomorphic via polar coordinates $(|z|,|w|,\arg(z),\arg(w))$
to the product
\[
\{(r_{1},r_{2})\in\mathbb{R}_{+}^{2}\mid r_{1}^{\frac{1-\beta}{\beta}}<r_{2}<r_{1}^{\frac{\beta}{1-\beta}}\}\times\{(s,t)\in\mathbb{T}^{2}\mid d_{S^{1}}(s+t,2\pi h/k)<\theta\}\subseteq\mathbb{R}_{+}^{2}\times\mathbb{T}^{2}
\]
shown in Figure~\ref{fig:Argument-components}. The modulus component
is simply connected and identical for all basins and the argument
component of $B_{h}$ is a $\theta$-neighbourhood of the central
curve $s+t\equiv2\pi h/k$, that is a ``ribbon'' winding around
the torus $\mathbb{T}^{2}$.

For $d>2$, the modulus component $W(\beta)\cap\mathbb{R}_{+}^{d}$
is still simply connected and the argument component of $B_{h}$ is
a $\theta$-neighbourhood of a central hypersurface in $\mathbb{T}^{d}$
given by $s_{1}+\cdots+s_{d}\equiv2\pi h/k$. For general $R>0$,
the basins are truncated, but remain the same near the origin and
preserve their homotopy type.
\end{rem}

\section{\label{sec:Elimination}Elimination of terms}

In this section, we will prove that under a Brjuno-type condition
we can holomorphically eliminate infinite families of monomials even
in the presence of resonances. We will then apply this to germs of
the form (\ref{eq:BZgerms}) to simplify the unknown tail.

Throughout this section, we will use multi-index notation:
\begin{notation}
Let $\alpha=(\alpha_{1},\ldots,\alpha_{d})\in\mathbb{N}^{d}$ be a
multi-index and $z=(z^{1},\ldots,z^{d})\in\mathbb{C}^{d}$. Then we
set $z^{\alpha}:=(z^{1})^{\alpha_{1}}\cdots(z^{d})^{\alpha_{d}}$
and $|\alpha|:=\alpha_{1}+\cdots+\alpha_{d}$. For $\alpha,\beta\in\mathbb{N}^{d}$,
we write $\alpha\le\beta$ if $\alpha_{j}\le\beta_{j}$ for $j=1,\ldots,d$.
\end{notation}

We first introduce the notion of a Brjuno set of exponents:
\begin{defn}
\label{def:BrjunoSet}Let $F$ be a germ of endomorphisms of $\mathbb{C}^{d}$
with
\[
dF_{0}=\Lambda=\diag(\lambda_{1},\ldots,\lambda_{d}).
\]
A set $A\subseteq\mathbb{N}^{d}$ is a \emph{Brjuno set} (\emph{of
exponents}) for $(\lambda_{1},\ldots,\lambda_{d})$ (or for $F$),
if
\begin{equation}
\sum_{k\ge1}2^{-k}\log\omega_{A}^{-1}(2^{k})<\infty,\label{eq:BrjunoSet}
\end{equation}
where
\begin{equation}
\omega_{A}(k):=\min\{|\lambda^{\alpha}-\lambda_{i}|\mid\alpha\in A,2\le|\alpha|\le k,1\le i\le d\}\cup\{1\}\label{eq:BrjunoFunction}
\end{equation}
for $k\ge2$.
\end{defn}

\begin{rem}
Subsets and finite unions of Brjuno sets are Brjuno sets.
\end{rem}

This definition includes the classical Brjuno condition from \cite{Brjuno1973Analyticalformofdifferentialequations}
and the partial Brjuno condition from \cite{Poeschel1986Oninvariantmanifoldsofcomplexanalyticmappingsnearfixedpoints}
as follows:
\begin{defn}
\label{def:BrjunoAndPart}Let $\lambda_{1},\ldots,\lambda_{d}\in\mathbb{C}$.
\begin{enumerate}
\item $\{\lambda_{1},\ldots,\lambda_{d}\}$ satisfies the \emph{Brjuno condition},
if $A=\mathbb{N}^{d}$ is a Brjuno set for $(\lambda_{1},\ldots,\lambda_{d})$.
\item $L\subseteq\{\lambda_{1},\ldots,\lambda_{d}\}$ satisfies the \emph{partial
Brjuno condition} (wrt. $(\lambda_{1},\ldots,\lambda_{d})$), if $A=\{\alpha\in\mathbb{N}^{d}\mid\alpha_{j}=0\text{ for }\lambda_{j}\notin L\}$
is a Brjuno set for $(\lambda_{1},\ldots,\lambda_{d})$.
\end{enumerate}
\end{defn}

\cite{Brjuno1973Analyticalformofdifferentialequations} and \cite{Poeschel1986Oninvariantmanifoldsofcomplexanalyticmappingsnearfixedpoints}
prove full and partial analytic linearisability on submanifolds tangent
to the union of the eigenspaces of the multipliers that satisfy the
respective condition. The following theorem generalises these results
in the context of eliminating infinite families of monomials. A different
generalisation aiming at full linearisation in the presence of resonances
has been explored in \cite{Raissy2011Brjunoconditionsforlinearizationinpresenceofresonances}.
\begin{thm}
\label{thm:GeneralElimination}Let $F$ be a germ of endomorphisms
of $\mathbb{C}^{d}$ of the form  $F(z)=\Lambda z+\sum_{|\alpha|>1}\sum_{j=1}^{d}f_{\alpha}^{j}z^{\alpha}e_{j}$
with $\Lambda=\diag(\lambda_{1},\ldots,\lambda_{d})$. Let $A_{0}$
and $A$ be disjoint sets of multi-indices in $\mathbb{N}^{d}$ such
that
\begin{enumerate}
\item \label{enu:Elim1}If $\alpha\in A_{0}$ and $\beta\le\alpha$, then
$\beta\in A_{0}$, and if $\alpha\in A$ and $\beta\le\alpha$, then
$\beta\in A_{0}\cup A$.
\item \label{enu:Elim2}If $\beta_{1},\ldots,\beta_{l}\in A_{0}$ and $\beta_{1}+\cdots+\beta_{l}\in A_{0}\cup A$,
$|\beta_{1}|\ge2$ and $f_{\beta_{1}}^{j_{1}}\cdots f_{\beta_{l}}^{j_{l}}\neq0$,
then $e_{j_{1}}+\cdots+e_{j_{l}}\notin A$.
\item \label{enu:Elim3Brjuno}$A$ is a Brjuno set for $F$.
\end{enumerate}
Then there exists a local biholomorphism $H\in\Aut(\mathbb{C}^{d},0)$
conjugating $F$ to $G=H^{-1}\circ F\circ H$ such that $G(z)=\sum g_{\alpha}z^{\alpha}$
with $g_{\alpha}=f_{\alpha}$ for $\alpha\in A_{0}$ and $g_{\alpha}=0$
for $\alpha\in A$.
\end{thm}

For $A_{0}=\{|\alpha|\le1\}$ and $A=\mathbb{N}^{d}\backslash A_{0}$
or $A=(\mathbb{N}^{m}\times\{0\})\backslash A_{0}$, we recover the
results from \cite{Brjuno1973Analyticalformofdifferentialequations}
and \cite{Poeschel1986Oninvariantmanifoldsofcomplexanalyticmappingsnearfixedpoints}.
A novelty of phrasing the result in this way is that it can be iterated
to obtain Theorem~\ref{thm:IteratedEliminationSimple}.
\begin{rem}
If we assume $\sum_{\alpha\in A_{0}}f_{\alpha}z^{\alpha}$ to be in
Poincaré-Dulac normal form, the condition $f_{\beta_{1}}^{j_{1}}\cdots f_{\beta_{l}}^{j_{l}}\neq0$
can be replaced by $\lambda^{\beta_{m}}\neq\lambda_{j_{m}}$ for $1\le m\le l$
to avoid dependence of Condition~(\ref{enu:Elim2}) on the specific
germ $F$.
\end{rem}

The proof of Theorem~\ref{thm:GeneralElimination} emerges largely
by careful examination of that in \cite{Poeschel1986Oninvariantmanifoldsofcomplexanalyticmappingsnearfixedpoints}
with some adjustments to avoid the assumption $\min_{1\le j\le d}|\lambda_{j}|\le1$
in the proofs of Lemmas~\ref{lem:Siegel} and \ref{lem:Brjuno}.
In \cite{Poeschel1986Oninvariantmanifoldsofcomplexanalyticmappingsnearfixedpoints}
this is ensured by considering $F^{-1}$ if necessary, but Condition~(\ref{enu:Elim2})
in our theorem is not invariant under taking inverses.
\begin{proof}
Formal series
\begin{align*}
G(z) & =\Lambda z+g(z)=\sum_{|\alpha|\ge1}g_{\alpha}z^{\alpha}\quad\text{and}\quad H(z)=z+h(z)=\sum_{|\alpha|\ge1}h_{\alpha}z^{\alpha}
\end{align*}
of the required form emerge as solutions to the homological equation
$F\circ H=H\circ G$. Comparing coefficients for $\alpha\in\mathbb{N}^{d}\backslash\{0\}$,
this means
\begin{equation}
(\lambda^{\alpha}\id-\Lambda)h_{\alpha}=f_{\alpha}-g_{\alpha}+\sum_{2\le k<|\alpha|}\sum_{j_{1}\le\cdots\le j_{k}}\sum_{\beta_{1}+\cdots+\beta_{k}=\alpha}(f_{e_{J}}h_{\beta_{1}}^{j_{1}}\cdots h_{\beta_{k}}^{j_{k}}-h_{e_{J}}g_{\beta_{1}}^{j_{1}}\cdots g_{\beta_{k}}^{j_{k}}),\label{eq:HomolExplicit}
\end{equation}
where $e_{J}:=e_{j_{1}}+\cdots+e_{j_{k}}$. Take $h_{\alpha}=0$ for
$\alpha\notin A$. Then for $\alpha\in A_{0}$, the first term in
the sum vanishes by Condition~(\ref{enu:Elim1}) and the second term
vanishes by Condition~(\ref{enu:Elim2}), so $g_{\alpha}=f_{\alpha}$.
For $\alpha\in A$, $\lambda^{\alpha}\id-\Lambda$ is invertible by
Condition~(\ref{enu:itElim3}) and the right hand side depends only
on $h$-terms with index of order less than $|\alpha|$. Hence (\ref{eq:HomolExplicit})
determines $h_{\alpha}$ uniquely by recursion and we obtain a formal
solution $H$ and hence $G=H^{-1}\circ F\circ H$.

To show that $H$ (and hence $G$) converges in some neighbourhood
of the origin, we have to show
\begin{equation}
\sup_{\alpha\in\mathbb{N}^{d}}\frac{1}{|\alpha|}\log\norm{h_{\alpha}}_{1}<\infty.\label{eq:convOfConjugation}
\end{equation}
We apply the majorant method first used by C.~L.~Siegel in \cite{Siegel1942Iterationofanalyticfunctions}
and improved in \cite{Brjuno1973Analyticalformofdifferentialequations}.
We may assume (up to scaling of variables) that $\norm{f_{\alpha}}_{1}\le1$
for all $|\alpha|\ge2$. Now for $\alpha\in A$ again by Condition~(\ref{enu:Elim2}),
the second term in the sum in (\ref{eq:HomolExplicit}) vanishes and
it follows
\begin{equation}
\norm{h_{\alpha}}_{1}\le d\cdot\norm{h_{\alpha}}_{\infty}\le d\cdot\varepsilon_{\alpha}^{-1}\sum_{2\le k\le|\alpha|}\sum_{\beta_{1}+\cdots+\beta_{k}=\alpha}\norm{h_{\beta_{1}}}_{1}\cdots\norm{h_{\beta_{k}}}_{1},\label{eq:HomolTermEstimate}
\end{equation}
where $\varepsilon_{\alpha}:=\min_{1\le i\le d}|\lambda^{\alpha}-\lambda_{i}|$.
We estimate (\ref{eq:HomolTermEstimate}) in two parts, one on the
number of summands, the other on their size. We define recursively
$\sigma_{1}=1$ and
\begin{equation}
\sigma_{r}:=d\sum_{k=2}^{r}\sum_{r_{1}+\cdots+r_{k}=r}\sigma_{r_{1}}\cdots\sigma_{r_{\nu}}\quad\text{for }r\ge2,\label{eq:SigmaForTermNumber}
\end{equation}
and $\delta_{e_{1}}=\cdots=\delta_{e_{d}}=1$, $\delta_{\alpha}=0$
for $\alpha\notin A\cup\{e_{1},\ldots,e_{d}\}$, and
\begin{align}
\delta_{\alpha} & \coloneqq\varepsilon_{\alpha}^{-1}\max_{\substack{\beta_{1}+\cdots+\beta_{k}=\alpha\\
k\ge2
}
}\delta_{\beta_{1}}\cdots\delta_{\beta_{k}}\quad\text{for }\alpha\in A.\label{eq:deltaForDivisors}
\end{align}
Then, by induction on $|\alpha|$, (\ref{eq:HomolTermEstimate}) implies
\begin{equation}
\norm{h_{\alpha}}_{1}\le\sigma_{|\alpha|}\delta_{\alpha}\label{eq:BoundHbySigmaDelta}
\end{equation}
for $\alpha\in A$. Hence to prove (\ref{eq:convOfConjugation}) it
suffices to prove estimates of the same type for $\sigma_{|\alpha|}$
and $\delta_{\alpha}$.

The estimates on $\sigma_{r}$ go back to \cite{Siegel1942Iterationofanalyticfunctions}
and \cite{Sternberg1961InfiniteLiegroupsandtheformalaspectsofdynamicalsystems}:
Let $\sigma(t):=\sum_{r=1}^{\infty}\sigma_{r}t^{r}$ and observe
\begin{align*}
\sigma(t) & =t+\sum_{r=2}^{\infty}\paren[{\Big}]{d\sum_{k=2}^{r}\sum_{r_{1}+\cdots+r_{k}=r}\sigma_{r_{1}}\cdots\sigma_{r_{k}}}t^{r}\\
 & =t+d\sum_{k=2}^{\infty}\paren[{\Big}]{\sum_{r=1}^{\infty}\sigma_{r}t^{r}}^{k}\\
 & =t+d\frac{\sigma(t)^{2}}{1-\sigma(t)}.
\end{align*}
Solving for $t$ and requiring $\sigma(0)=0$ yields a unique holomorphic
solution
\[
\sigma(t)=\frac{1+t-\sqrt{(1+t)^{2}-4(d+1)t}}{2(d+1)}
\]
for small $t$, so $\sigma$ converges near $0$ and we have
\begin{equation}
\sup_{r\ge1}\frac{1}{r}\log\sigma_{r}<\infty.\label{eq:BoundSigmaR}
\end{equation}

The estimates on $\delta_{\alpha}$ take care of the small divisors
$\varepsilon_{\alpha}$ and proceed essentially like \cite{Brjuno1973Analyticalformofdifferentialequations}.
For every $|\alpha|\ge2$, we choose a maximising decomposition $\beta_{1}+\cdots+\beta_{k}=\alpha$
in (\ref{eq:deltaForDivisors}) such that
\begin{equation}
\delta_{\alpha}=\varepsilon_{\alpha}^{-1}\delta_{\beta_{1}}\cdots\delta_{\beta_{k}}\label{eq:ChosenDecompDelta}
\end{equation}
and $|\alpha|>|\beta_{1}|\ge\cdots\ge|\beta_{k}|\ge1$. In this way,
starting from $\delta_{\alpha}$ we proceed to decompose $\delta_{\beta_{1}},\ldots,\delta_{\beta_{k}}$
in the same way and continue the process until we arrive at a well-defined
decomposition of the form
\begin{equation}
\delta_{\alpha}=\varepsilon_{\alpha_{0}}^{-1}\varepsilon_{\alpha_{1}}^{-1}\cdots\varepsilon_{\alpha_{s}}^{-1},\label{eq:DeltaAsAllEps}
\end{equation}
where $2\le|\alpha_{s}|,\ldots,|\alpha_{1}|<|\alpha_{0}|$ and $\alpha_{0}=\alpha$.
We may further choose an index $i_{\alpha}\in\{1,\ldots,d\}$ for
each $|\alpha|\ge2$ such that
\[
\varepsilon_{\alpha}=|\lambda^{\alpha}-\lambda_{i_{\alpha}}|.
\]
Let $n\in\mathbb{N}$ such that
\[
n-1\ge2\min_{1\le i\le d}|\lambda_{i}|
\]
and $\theta>0$ such that
\begin{equation}
n\theta=\min_{1\le i\le d}\{|\lambda_{i}|,1\}\le1.\label{eq:ThetaDef}
\end{equation}
Now for the indices $\alpha_{0},\ldots,\alpha_{s}$ in the decomposition
(\ref{eq:DeltaAsAllEps}), we want to bound
\[
N_{m}^{j}(\alpha):=\#\{l\in\{0,\ldots,s\}\mid i_{\alpha_{l}}=j,\varepsilon_{\alpha_{l}}<\theta\omega_{A}(m)\},
\]
for $j\le d$ and $m\in\mathbb{N}$, where we adopt the convention
$\omega_{A}(1):=+\infty$. First we need the following lemma, that
\cite{Poeschel1986Oninvariantmanifoldsofcomplexanalyticmappingsnearfixedpoints}
attributes to Siegel, showing that indices contributing to $N_{m}^{j}(\alpha)$
cannot be to close to each other:
\begin{lem}[Siegel]
\label{lem:Siegel}Let $m\ge1$. If $\alpha>\beta$ are such that
$\varepsilon_{\alpha}<\theta\omega_{A}(m),$ $\varepsilon_{\beta}<\theta\omega_{A}(m),$
and $i_{\alpha}=i_{\beta}=j,$ then $|\alpha-\beta|\ge m$.
\end{lem}

\begin{subproof}For $m=1$, $\alpha>\beta$ implies $|\alpha-\beta|\ge1$.
For $m\ge2$, by the definition of $\omega_{A}$ in (\ref{eq:BrjunoFunction}),
we have $\omega_{A}(m)\le1$. With that and (\ref{eq:ThetaDef}),
the hypothesis $\varepsilon_{\beta}<\theta\omega_{A}(m)$ implies
\[
|\lambda^{\beta}|>|\lambda_{j}|-\theta\omega_{A}(m)\ge n\theta-\theta=(n-1)\theta
\]
and hence
\begin{align*}
2\theta\omega(m) & >\varepsilon_{\alpha}+\varepsilon_{\beta}\\
 & =|\lambda^{\alpha}-\lambda_{j}|+|\lambda^{\beta}-\lambda_{j}|\\
 & \ge|\lambda^{\alpha}-\lambda^{\beta}|\\
 & =|\lambda^{\beta}||\lambda^{\alpha-\beta}-1|\\
 & >(n-1)\theta\cdot(\min_{1\le i\le d}|\lambda_{i}|)^{-1}\omega(|\alpha-\beta|+1)\\
 & \ge2\theta\omega(|\alpha-\beta|+1),
\end{align*}
i.e.\ $\omega(m)>\omega(|\alpha-\beta|+1)$. But $\omega$ is decreasing,
so we must have $|\alpha-\beta|\ge m$.\end{subproof}

We can now show Brjuno's estimate on $N_{m}^{j}(\alpha)$:
\begin{lem}[Brjuno, \cite{Brjuno1973Analyticalformofdifferentialequations}]
\label{lem:Brjuno}For $|\alpha|\ge2$, $m\ge1$, and $1\le j\le d$,
we have
\[
N_{m}^{j}(\alpha)\le\begin{cases}
0, & \text{for }|\alpha|\le m\\
2|\alpha|/m-1, & \text{for }|\alpha|>m.
\end{cases}
\]
\end{lem}

\begin{subproof}We fix $m$ and $j$ and proceed by induction on
$|\alpha|$.

If $2\le|\alpha|\le m$, we have
\[
\varepsilon_{\alpha_{l}}\ge\omega_{A}(|\alpha|)\ge\omega_{A}(m)\ge\theta\omega_{A}(m)
\]
for all $0\le l\le s$, so $N_{m}^{j}(\alpha)=0.$

If $|\alpha|>m$, we take the chosen decomposition (\ref{eq:ChosenDecompDelta})
and note that only $|\beta_{1}|$ may be greater than $K=\max\{|\alpha|-m,m\}$.
If $|\beta_{1}|>K$, we decompose $\delta_{\beta_{1}}$ in the same
way and repeat this at most $m-1$ times to obtain a decomposition
\begin{equation}
\delta_{\alpha}=\varepsilon_{\alpha}^{-1}\varepsilon_{\alpha_{1}}^{-1}\cdots\varepsilon_{\alpha_{k}}^{-1}\cdot\delta_{\beta_{1}}\cdots\delta_{\beta_{l}}\label{eq:DeltaDecompForInd}
\end{equation}
with $0\le k\le m-1$, $l\ge2$ and
\begin{align}
\alpha & >\alpha_{1}>\cdots>\alpha_{k}\nonumber \\
\alpha & =\beta_{1}+\cdots+\beta_{l}\nonumber \\
|\alpha_{k}| & >K\ge|\beta_{1}|\ge\cdots\ge|\beta_{l}|.\label{eq:deltaDecompIndProp}
\end{align}
In particular, (\ref{eq:deltaDecompIndProp}) implies $|\alpha-\alpha_{k}|<m$.
Hence Lemma~\ref{lem:Siegel} shows that at most one of the $\varepsilon$-
factors in (\ref{eq:DeltaDecompForInd}) can contribute to $N_{m}^{j}(\alpha)$
and we have
\[
N_{m}^{j}(\alpha)\le1+N_{m}^{j}(\beta_{1})+\cdots+N_{m}^{j}(\beta_{l}).
\]
Now let $0\le h\le l$ such that be the such that $|\beta_{1}|,\ldots,|\beta_{h}|>m\ge|\beta_{h+1}|,\ldots,|\beta_{l}|$.
Then by (\ref{eq:deltaDecompIndProp}), we have $|\beta_{1}|,\ldots,|\beta_{h}|\le|\alpha|-m$
and, by induction, the terms with $|\beta|\le m$ vanish and we have
\begin{align*}
N_{m}^{j}(\alpha) & \le1+N_{m}^{j}(\beta_{1})+\cdots+N_{m}^{j}(\beta_{h})\\
 & \le1+2|\beta_{1}+\cdots+\beta_{h}|/m-h\\
 & \le\begin{cases}
1, & \text{for }h=0\\
2\frac{|\alpha|-m}{m}, & \text{for }h=1\\
2|\alpha|/m-(h-1), & \text{for }h\ge2
\end{cases}\\
 & \le2|\alpha|/m-1.\qedhere
\end{align*}
\end{subproof}

To estimate the product (\ref{eq:DeltaAsAllEps}) we partition the
indices into sets
\[
I_{l}:=\{0\le k\le s\mid\theta\omega_{A}(2^{l+1})\le\varepsilon_{\alpha_{k}}<\theta\omega_{A}(2^{l})\}\quad\text{for }l\ge0
\]
(recall for $I_{0}$ the convention $\omega_{A}(1)=+\infty$). By
Lemma~\ref{lem:Brjuno}, we have
\[
\#I_{l}\le N_{2^{l}}^{1}(\alpha)+\cdots+N_{2^{l}}^{d}(\alpha)\le2d|\alpha|2^{-l}
\]
and we can estimate
\begin{align*}
\frac{1}{|\alpha|}\log\delta_{\alpha} & =\sum_{k=0}^{s}\frac{1}{|\alpha|}\log\varepsilon_{\alpha_{k}}^{-1}\\
 & \le\sum_{l\ge0}\sum_{k\in I_{l}}\frac{1}{|\alpha|}\log(\theta^{-1}\omega_{A}^{-1}(2^{l+1}))\\
 & \le2d\sum_{l\ge0}2^{-l}\log(\theta^{-1}\omega_{A}^{-1}(2^{l+1}))\\
 & =4d\log(\theta^{-1})+4d\sum_{l\ge1}2^{-l}\log(\omega_{A}^{-1}(2^{l})).
\end{align*}
This bound is independent of $\alpha$ and, since $A$ is a Brjuno
set, it is finite. Hence with (\ref{eq:BoundHbySigmaDelta}) and (\ref{eq:BoundSigmaR})
it follows that
\[
\sup_{\alpha\in\mathbb{N}^{d}}\frac{1}{|\alpha|}\log\norm{h_{\alpha}}\le\sup_{\alpha\in\mathbb{N}^{d}}\frac{1}{|\alpha|}\log\delta_{\alpha}+\sup_{r\ge1}\frac{1}{r}\log\sigma_{r}<\infty
\]
and thus $H$ and $G$ converge.
\end{proof}
\begin{proof}[Proof of Theorem~\ref{thm:IteratedEliminationSimple}]
Assume by induction on $k_{0}$, that $f_{\alpha}=0$ for $\alpha\in A\backslash A_{k_{0}}$.
We show that $A_{0}'=A_{\overline{k_{0}-1}}$ and $A'=A_{k_{0}}$
satisfy the prerequisites of Theorem~\ref{thm:GeneralElimination}.

Conditions~(\ref{enu:Elim1}) and (\ref{enu:Elim3Brjuno}) follow
directly from their counterparts.

Let $\beta_{1},\ldots,\beta_{l}\in A_{0}'=A_{\overline{k_{0}-1}}$
as in Condition~(\ref{enu:Elim2}). By induction $f_{\beta_{1}}^{j_{1}}\cdots f_{\beta_{l}}^{j_{l}}\neq0$
implies $\beta_{1},\ldots,\beta_{l}\in A_{0}$, so Assumption~(\ref{enu:itElim2})
of Theorem~\ref{thm:IteratedEliminationSimple} implies $e_{J}\notin A_{k}=A'$,
and Condition~(\ref{enu:Elim2}) is satisfied.

Therefore Theorem~\ref{thm:GeneralElimination} shows that $F$ is
conjugate to $G$ with $g_{\alpha}=f_{\alpha}$ for $\alpha\in A_{0}$
and $g_{\alpha}=0$ for $\alpha\in A$.
\end{proof}

\subsection{The one-resonant case}

In the one-resonant case, the classical Brjuno condition on subsets
already implies much more:
\begin{lem}
\label{lem:BrjunoImpliesAll}If $F\in\Aut(\mathbb{C}^{d},0)$ is one-resonant
of index $\alpha=(1,\ldots,1)$ at $0$, then the following are equivalent:
\begin{enumerate}
\item \label{enu:Brj}$\{\lambda_{1},\ldots,\lambda_{d}\}\backslash\{\lambda_{j}\}$
satisfies the Brjuno condition for every $j\le d$.
\item \label{enu:pBrj}$\{\lambda_{1},\ldots,\lambda_{d}\}\backslash\{\lambda_{j}\}$
satisfies the partial Brjuno condition wrt.\ $(\lambda_{1},\ldots,\lambda_{d})$
for every $j\le d$.
\item \label{enu:BrjSets}$A_{k}=\{\beta\in\mathbb{N}^{d}\mid|\beta|>kd+1,\min\{\beta_{1},\ldots,\beta_{d}\}=k\}$
is a Brjuno set for $F$ for every $k\in\mathbb{N}$.
\end{enumerate}
\end{lem}

\begin{proof}
The relevant minimal divisors for Items~(\ref{enu:Brj}) and (\ref{enu:pBrj})
are
\[
\omega_{j}(l):=\min\{|\lambda^{\alpha}-\lambda_{i}|\mid2\le|\alpha|\le l,\alpha_{j}=0,i\neq j\}\cup\{1\}
\]
and $\omega_{M_{j}}(l)$ with $M_{j}:=\{\beta\in\mathbb{N}^{d}\mid\beta_{j}=0\}$,
respectively, for $j\le d$ and $l\ge2$.

\emph{}%
\mbox{%
\emph{(\ref{enu:BrjSets}) $\Rightarrow$ (\ref{enu:pBrj}) $\Rightarrow$
(\ref{enu:Brj})}%
} follows from $\omega_{A_{0}}\le\omega_{M_{j}}\le\omega_{j}$.

\emph{}%
\mbox{%
\emph{(\ref{enu:Brj}) $\Rightarrow$ (\ref{enu:pBrj})}%
}\emph{.} Fix $j\le d$ and let $m_{0}\ge2$ large enough that $2^{m_{0}}\ge d-1$.
Then for $m\ge m_{0}$ the only divisors contributing to $\omega_{M_{j}}(2^{m})$,
but not to $\omega_{j}(2^{m})$, are of the form
\[
|\lambda^{\beta}-\lambda_{j}|=|\lambda_{j}||\lambda^{\beta+\alpha-e_{j}}-1|\ge|\lambda_{j}|\max_{i\neq j}(|\lambda_{i}|^{-1})\omega_{j}(2^{m}+d-1)\ge\theta\omega_{j}(2^{m+1}),
\]
where $\theta:=|\lambda_{j}|\max_{i\neq j}|\lambda_{i}|^{-1}$. Hence,
we have
\begin{align*}
\sum_{m\ge m_{0}}2^{-m}\log\omega_{M_{j}}^{-1}(2^{m}) & \le\sum_{m\ge m_{0}}2^{-m}\log(\theta^{-1}\omega_{j}^{-1}(2^{m+1}))\\
 & \le2\log(\theta^{-1})+2\sum_{m>m_{0}}2^{-m}\log(\omega_{j}^{-1}(2^{m})).
\end{align*}
If we have (\ref{enu:Brj}) this is finite for each $j\le d$, implying
(\ref{enu:pBrj}).

\emph{}%
\mbox{%
\emph{(\ref{enu:pBrj}) $\Rightarrow$ (\ref{enu:BrjSets})}%
}. For $\beta\in A_{k}$, we have $|\beta-k\alpha|\ge2$ and $\beta_{j}=k$
for some $j\le d$, so $(\beta-k\alpha)_{j}=0$ and for any $i\le d$,
we have
\[
|\lambda^{\beta}-\lambda_{i}|=|\lambda^{\beta-k\alpha}-\lambda_{i}|\ge\omega_{M_{j}}(|\beta-k\alpha|).
\]
 Hence
\begin{align*}
\sum_{m\ge1}2^{-m}\log\omega_{A_{k}}^{-1}(2^{m}) & \le\sum_{m\ge1}2^{-m}\max_{1\le j\le d}(\log\omega_{M_{j}}^{-1}(2^{m}-k\alpha))\\
 & \le\sum_{j=1}^{d}\sum_{m\ge1}2^{-m}\log\omega_{M_{j}}^{-1}(2^{m}).
\end{align*}
If we have (\ref{enu:pBrj}), this is finite, implying (\ref{enu:BrjSets}).
\end{proof}
With this, we only need the weakest assumption (\ref{enu:Brj}) to
show that we can assume the tail of our map to be of a nicer form:
\begin{cor}
\label{cor:NicerTail}Let $F$, $\lambda_{1},\ldots,\lambda_{d}$
and $l\in\mathbb{N}$ be as in Theorem~\ref{thm:BZ} such that $\{\lambda_{1},\ldots,\lambda_{d}\}\backslash\{\lambda_{j}\}$
satisfies the Brjuno condition for every $j\le d$, and let $\alpha=(1,\ldots,1)$.
Then there exist a local change of coordinates $\chi(z)=z+O(\norm z^{l})$
conjugating $F$ to
\[
z\mapsto F_{{\rm N}}(z)+O(z^{l\alpha}).
\]
\end{cor}

\begin{proof}
First, observe that, since $l>2k+1$, \cite[Theorem~3.6]{BracciZaitsev2013Dynamicsofoneresonantbiholomorphisms}
implies that $F$ is conjugate to $G$ with $G(z)=F_{{\rm N}}(z)+O(\norm z^{ld+2})$.
We want to apply Theorem~\ref{thm:IteratedEliminationSimple} to
$G$, $A_{0}=\{|\beta|\le ld+1\}$ and
\[
A_{m}=\{\beta\in\mathbb{N}^{d}\mid|\beta|>ld+1,\min\{\beta_{1},\ldots,\beta_{d}\}=m-1\}
\]
 for $1\le m\le l+1$. Condition~(\ref{enu:itElim1}) is clear. For
$\beta\in A_{0}$ with $f_{\beta}^{j}\neq0$, we have $\beta=\epsilon\alpha+e_{j}$
with $\epsilon=0,1$, so $\beta\ge e_{j}$ and $\beta\ge\alpha+e_{j}$
if $|\beta|\ge2$. Hence for $\beta_{1},\ldots,\beta_{r}\in A_{0}$
with $|\beta_{1}|\ge2$ such that $f_{\beta_{1}}^{j_{1}}\cdots f_{\beta_{r}}^{j_{r}}\neq0$,
we have
\begin{equation}
e_{J}=e_{j_{1}}+\cdots+e_{j_{r}}\le\beta_{1}+\cdots+\beta_{r}-\alpha.\label{eq:eJLessBeta}
\end{equation}
If $\beta_{1}+\cdots+\beta_{r}\in A_{0}$, then this implies $e_{J}\in A_{0}$.
If $\beta_{1}+\cdots+\beta_{r}\in A_{m}$, then for some $i\le d$,
we have
\[
m=(\beta_{1}+\cdots+\beta_{r})_{i}\ge(e_{J})_{i}+1
\]
by (\ref{eq:eJLessBeta}). So in both cases, $e_{J}\notin A_{m'}$
for any $m'\ge\max\{1,m\}$ and Condition~(\ref{enu:itElim2}) is
satisfied. Condition~(\ref{enu:itElim3}) follows from Lemma~\ref{lem:BrjunoImpliesAll}.
Now Theorem~\ref{thm:IteratedEliminationSimple} shows that $G$
is locally conjugate to $H$ such that $H(z)=F_{{\rm N}}(z)+R(z)$,
where $R(z)$ only contains monomials $z^{\beta}$ with $\beta\in\mathbb{N}_{0}^{d}\backslash A_{\overline{l+1}}=\{\beta\ge l\alpha\}$
or $R(z)=O(z^{l\alpha})$.
\end{proof}

\section{\label{sec:Stable-Orbits}Classification of stable orbits}

In this section, under the Brjuno condition on subsets, we identify
all stable orbits of our germs near the fixed point and conclude that
the global basins corresponding to our local basins are (unions of)
Fatou components.

Corollary~\ref{cor:NicerTail} implies immediately that there exist
rotating stable orbits that do not converge to the origin:
\begin{cor}
Let $F\in\Aut(\mathbb{C}^{d},0)$ be as in Theorem~\ref{thm:BZ},
i.e.\ $F(z)=F_{{\rm N}}(z)+O(\norm z^{l})$, where
\[
F_{{\rm N}}(z)=\Lambda z\cdot\paren [\Big]{1-\frac{z^{k\alpha}}{kd}}.
\]
Assume further that $\{\lambda_{1},\ldots,\lambda_{d}\}\backslash\{\lambda_{j}\}$
satisfies the Brjuno condition for every $j\le d$. Then $F$ admits
Siegel hypersurfaces $D_{1},\ldots,D_{d}$ tangent to $\{z_{1}=0\},\cdots,\{z_{d}=0\}$
respectively.
\end{cor}

\begin{proof}
Let $(\chi_{1},\ldots,\chi_{d})(z)=\chi(z)=z+O(\norm z^{l})$ be as
in Corollary~\ref{cor:NicerTail}. Then on $D_{j}:=\{\chi_{j}(z)=0\}$
for $j=1,\ldots,d$, $F$ acts as the irrational rotation $w\mapsto\Lambda w$.
\end{proof}
In fact, using Corollary~\ref{cor:NicerTail}, we can extend the
proof of \cite[Lemma~2.5]{BracciRaissyStensonesAutomorphismsofmathbbCkwithaninvariantnonrecurrentattractingFatoucomponentbiholomorphictomathbbCtimesmathbbCk1}
to classify all stable orbits near the origin:
\begin{prop}
\label{prop:stableOrbitBehaviour}Let $F$ and $B_{1},\ldots,B_{k-1}$
be as in Theorem~\ref{thm:BZ} such that $\{\lambda_{1},\ldots,\lambda_{d}\}\backslash\{\lambda_{j}\}$
satisfies the Brjuno condition for every $j\le d$. For $z\in\mathbb{C}^{d}$
let $z_{n}:=F^{\circ n}(z)$ and $u_{n}:=z_{n}^{\alpha}$. Then there
exists $r>0$ such that: if $z_{n}\in B_{r}(0)$ eventually, then
either $\{z_{n}\}_{n}$ is contained in one of the Siegel hypersurfaces
$D_{1},\ldots,D_{d}$, or $z_{n}\to0$ and there exists a unique $h\in\{0,\ldots,k-1\}$
such that
\begin{enumerate}
\item \label{enu:OB1un}$u_{n}\sim e^{2\pi ih/k}n^{-\expfrac 1k}$ (i.e.\ $\lim_{n\to\infty}\sqrt[k]{n}u_{n}=e^{2\pi ih/k}$),
in particular, $|u_{n}|\sim n^{-1/k}$,
\item \label{enu:OB2zn}$|z_{n}^{j}|\approx n^{-\expfrac 1{kd}}$ for $j=1,\ldots,d$,
\item \label{enu:OB3Bh}for every $R>0,\theta\in(0,\nicefrac{\pi}{2k}),$
and $\beta\in(0,1/d)$ with $\beta(l+1)>1$, we have $z_{n}\in B_{h}(R,\theta,\beta)$
eventually (in particular, $z_{n}\in B_{h}$ eventually).
\item \label{enu:OB4uniformity}The upper bounds $|u_{n}|\le n^{-1/k}(1+o(1))$
and $|z_{n}^{j}|=O(n^{-\expfrac 1{kd}})$ in (\ref{enu:OB1un}) and
(\ref{enu:OB2zn}) are uniform in $B_{h}$.
\end{enumerate}
\end{prop}

\begin{rem}
In particular, Part~(\ref{enu:OB4uniformity}) shows that ${\pbrace{F^{n}}}_{n\in\mathbb{N}}$
is normal on each local basin $B_{h}$, $h=0,\ldots,k-1$, hence $B_{h}$
is contained in a Fatou component for $F$.
\end{rem}

In the proof, we will use the following result of \cite[Lemma~5.3]{BracciRaissyStensonesAutomorphismsofmathbbCkwithaninvariantnonrecurrentattractingFatoucomponentbiholomorphictomathbbCtimesmathbbCk1}:
\begin{lem}
\label{lem:CoordsAndBasins}Let $\beta\in(0,1/d)$ and $l\in\mathbb{N}$
such that $\beta(l+1)>1$. For every germ of biholomorphisms $\chi\in\Aut(\mathbb{C}^{d},0)$
with $\chi(z)=z+O(\norm z^{l})$ and every $\beta'\in(0,\beta)$,
there exists $\varepsilon>0$ such that
\[
\chi(W(\beta)\cap\{\norm z<\varepsilon\})\subseteq W(\beta').
\]
\begin{proof}[Proof of Proposition~\ref{prop:stableOrbitBehaviour}]
First assume
\[
F(z)=F_{{\rm N}}(z)+O(z^{l\alpha}).
\]
Then $u_{n}=0$ for some $n\in\mathbb{N}$, if and only if $z_{n}\in D_{j}=\{z_{j}=0\}$
for some $j$ and hence the whole orbit is contained in $D_{j}$.
Now assume $u_{n}\neq0$ for all $n\in\mathbb{N}$ and we can define
$U:=u^{-k}$ and $U_{n}:=u_{n}^{-k}$. Then
\begin{equation}
U_{n+1}=U_{n}+1+O(U_{n}^{-1},U_{n}^{1-(l-1)/k})\quad\text{for all }n\in\mathbb{N}.\label{eq:Un+1}
\end{equation}
Since $l>2k+1$, there exists $r>0$ such that for $z_{n}\in B_{r}(0)$,
we have
\[
|U_{n+1}-U_{n}-1|<\frac{1}{2}.
\]
So whenever $z_{n}\in B_{r}(0)$ eventually, we have $|U_{n}|\to\infty$.
Hence in this case (\ref{eq:Un+1}) shows that for any $c>1$ there
exists $n_{c}\in\mathbb{N}$ such that $|U_{n+1}-U_{n}-1|<(c-1)/c$
for all $n\ge n_{c}$. By induction for all $n\ge n_{c}$, we have
\begin{equation}
\Re U_{n}\ge\Re U_{n_{c}}+\frac{n-n_{c}}{c}\label{eq:ReUnLowerBound}
\end{equation}
and
\[
|U_{n}|\le|U_{n_{c}}|+c(n-n_{c}).
\]
 For $c\searrow1$, it follows
\begin{equation}
\lim_{n\to\infty}\frac{\Re U_{n}}{n}=\lim_{n\to\infty}\frac{|U_{n}|}{n}=1,\label{eq:LimitsUk}
\end{equation}
hence
\begin{equation}
\lim_{n\to\infty}nu_{n}^{k}=1.\label{eq:Limitsuk}
\end{equation}
Fix $j\in\{1,\ldots,d\}$. By induction on $n\ge1$ we have
\begin{equation}
z_{n}^{j}=z^{j}\lambda_{j}^{n}\prod_{i=0}^{n-1}\paren[{\Big}]{1-\frac{u_{i}^{k\alpha}}{kd}}+\sum_{i=0}^{n-1}R^{j}(z_{i})\prod_{\nu=i+1}^{n-1}\lambda_{j}\paren[{\Big}]{1-\frac{u_{\nu}^{k\alpha}}{kd}},\label{eq:znj-expression}
\end{equation}
where $R^{j}(z)=O(z^{\alpha l})=O(u^{l})$. From (\ref{eq:Limitsuk}),
it follows that as $i\to\infty$, we have
\[
\log\paren[{\Big}]{1-\frac{u_{i}^{k}}{kd}}\sim-\frac{u_{i}^{k}}{kd}\sim-\frac{1}{ikd}.
\]
Therefore as $n\to\infty$, we have
\begin{align*}
\prod_{i=0}^{n-1}\paren[{\Big}]{1-\frac{u_{i}^{k}}{kd}} & =\exp\paren[{\Big}]{\sum_{i=0}^{n-1}\log{{\paren[{\Big}]{1-\frac{u_{i}^{k}}{kd}}}}}\\
 & \approx\exp\paren[{\Big}]{-\frac{1}{kd}\sum_{i=0}^{n-1}\frac{1}{i}}\\
 & \approx n^{-\expfrac 1{kd}}
\end{align*}
and
\[
\prod_{\nu=i+1}^{n-1}\paren[{\Big}]{1-\frac{u_{\nu}^{k}}{kd}}\approx\frac{i^{\expfrac 1{kd}}}{n^{\expfrac 1{kd}}}.
\]
With this and since $R^{j}(z_{i})=O(u_{i}^{l})=O(i^{-l/k})$, (\ref{eq:znj-expression})
implies
\begin{align*}
|z_{n}^{j}| & \approx n^{-\expfrac 1{kd}}\abs[{\Big}]{1+\sum_{i=0}^{n-1}O(i^{\expfrac 1{kd}-\expfrac lk})}\approx n^{-\expfrac 1{kd}},
\end{align*}
since $l>k+1$ and hence $\frac{1}{kd}-\frac{l}{k}<\frac{1}{kd}-1-\frac{1}{k}<-1$.
This proves Part~(\ref{enu:OB2zn}).
\end{proof}
\end{lem}

\begin{proof}
Let now $R>0,\theta\in(0,\nicefrac{\pi}{2k}),\beta\in(0,1/d)$. Then
for $n\to\infty$, Part~(\ref{enu:OB2zn}) implies
\[
|z_{n}^{j}|\approx n^{-\expfrac 1{kd}}\sim|u_{n}|^{1/d}=o(|u_{n}|^{\beta}),
\]
so $z_{n}\in W(\beta)$ eventually, and by (\ref{eq:LimitsUk}) we
have $U_{n}\in H(R,\theta)$ for large enough $n$. In particular
for $(R,\theta,\beta)=(R_{0},\theta_{0},\beta_{0})$, this means $z_{n}\in B_{0}\cup\cdots\cup B_{k-1}$
eventually, but each $B_{h}$ is $F$-invariant by Theorem~\ref{thm:BZ},
so $z_{n}$ stays in one unique $B_{h}$. Hence $u_{n}$ stays in
the image of the unique branch of the $k$-th root centred around
$\exp\paren [\big]{\frac{2\pi ih}{k}}$, implying Part~(\ref{enu:OB3Bh}),
and we can extract the $k$-th root from (\ref{eq:Limitsuk}) to get
Part~(\ref{enu:OB1un}).

To show Part~(\ref{enu:OB4uniformity}), we recall that in the proof
of \cite[Theorem~1.1]{BracciZaitsev2013Dynamicsofoneresonantbiholomorphisms},
that implies Theorem~\ref{thm:BZ}, $R_{0}$, $\theta_{0}$, and
$\beta_{0}$ are chosen such that
\[
|U_{1}-U-1|<\frac{1}{2}\quad\text{for all }U\in H(R_{0},\theta_{0}).
\]
Hence in (\ref{eq:ReUnLowerBound}), $n_{c}$ can be chosen in a uniform
manner and, since $\Re U_{n_{c}}>R_{0}$, we get uniform lower bounds
on $|U_{n}|>\Re U_{n}$. This becomes a uniform upper bound on the
convergence in (\ref{eq:Limitsuk}) and the subsequent estimates on
$|z_{n}|$.

For general $F$, Corollary~\ref{cor:NicerTail} shows that $F$
is locally conjugate to $z\mapsto F_{{\rm N}}(z)+O(z^{l\alpha})$
via a change of coordinates of the form $\chi(z)=z+O(\norm z^{l})$.
This clearly preserves Part~(\ref{enu:OB1un}) and (\ref{enu:OB2zn}),
and by Lemma~\ref{lem:CoordsAndBasins} $\chi$ preserves Part~(\ref{enu:OB3Bh})
and (\ref{enu:OB4uniformity}) as well.
\end{proof}
\begin{rem}
\label{rem:stableOrbitBehaviourInBasins}Without the Brjuno condition
on subsets, if $z\in\mathbb{C}^{d}$ is such that $z_{n}\to0$ and
$z_{n}\in W(\beta_{1})$ eventually for some $\beta_{1}\in(0,1/d)$
such that $\beta_{1}(l+d-1)>k+1$, we have
\[
U_{n+1}=U_{n}+1+O(U_{n}^{-1},U_{n}^{1-\frac{\beta_{1}(l+d-1)-1}{k}})
\]
with $1-\frac{\beta_{1}(l+d-1)-1}{k}<0$ and Part~(\ref{enu:OB1un})
through (\ref{enu:OB4uniformity}) of Proposition~\ref{prop:stableOrbitBehaviour}
still follow for these orbits in the same manner (cf.\ \cite[Lemma~2.5]{BracciRaissyStensonesAutomorphismsofmathbbCkwithaninvariantnonrecurrentattractingFatoucomponentbiholomorphictomathbbCtimesmathbbCk1}).
However, in this case we have not been able to determine the containing
Fatou components, as both the methods of the next section and of \cite[Section~5]{BracciRaissyStensonesAutomorphismsofmathbbCkwithaninvariantnonrecurrentattractingFatoucomponentbiholomorphictomathbbCtimesmathbbCk1}
rely on the Brjuno condition on subsets (cf.\ remarks in \cite[Section~1]{BracciRaissyStensonesAutomorphismsofmathbbCkwithaninvariantnonrecurrentattractingFatoucomponentbiholomorphictomathbbCtimesmathbbCk1}).
\end{rem}

\subsection{Global basins are Fatou components}

We show that the global basins corresponding to our local basins are
Fatou components and conclude the proof of the first part of Theorem~\ref{thm:FatouCCstarMulti}.
\begin{defn}
Let $F\in\Aut(\mathbb{C}^{d})$ be as in Theorem~\ref{thm:BZ}. Then
for $h=0,\ldots,k-1$, the \emph{global basin} corresponding to the
local basin $B_{h}$ is
\[
\Omega_{h}:=\bigcup_{n\in\mathbb{N}}F^{-n}(B_{h})
\]
and contains all points $z\in\mathbb{C}^{d}$ such that $F^{n}(z)\in B_{h}$
eventually.
\end{defn}

\begin{rem}
\label{rem:GlobalBasinProperties}The global basins $\Omega_{0},\ldots,\Omega_{k-1}$
are growing unions of preimages of $B_{0},\ldots,B_{k-1}$. As such
they are still pairwise disjoint, open, invariant and locally uniformly
attracted to $0$ under $F$. In particular However, unless $F$ is
a global automorphism, they may no longer be connected.
\end{rem}

\begin{cor}
\label{thm:IfPoeschelThenFatou}Let $F\in\Aut(\mathbb{C}^{d})$ as
in Theorem~\ref{thm:BZ} such that $\{\lambda_{1},\ldots,\lambda_{d}\}\backslash\{\lambda_{j}\}$
satisfies the Brjuno condition for each $j=1,\ldots,d$. Then the
connected components of $\Omega_{0},\ldots,\Omega_{k-1}$ are Fatou
components.
\end{cor}

\begin{proof}
Let $h\in\{0,\ldots,k-1\}$. By Remark~\ref{rem:GlobalBasinProperties},
each connected component $C$ of $\Omega_{h}$ is contained in a Fatou
component $V$. By normality, for each $z\in V$, we have $z_{n}\to0$,
but by Proposition~\ref{prop:stableOrbitBehaviour} that means $z_{n}$
is contained in $B_{h}$ eventually, i.e.\ $z\in\Omega_{h}$. Therefore
$C\subseteq V\subseteq\Omega_{h}$ and since $V$ is connected, it
follows that $V=C$.
\end{proof}
The first part of Theorem~\ref{thm:FatouCCstarMulti} now follows
from Corollary~\ref{thm:IfPoeschelThenFatou} and Proposition~\ref{prop:stableOrbitBehaviour}.

\section{\label{sec:Internal-geometry}Internal dynamics and geometry}

In this section we fix $h\in\{0,\ldots,k-1\}$ and introduce two closely
related systems of coordinates compatible with the action of $F$
on the local basin. One allows us to study the behaviour of orbits
under $F$, and the other extends to a biholomorphism of the corresponding
global basin to $\mathbb{C}\times(\mathbb{C}^{*})^{d-1}$. Note that
we do not assume the Brjuno condition on subsets in this section.

\subsection{\label{subsec:Fatou-coordinates}Fatou coordinates}

We define special coordinates that codify the dynamics of $F$ on
$B_{h}$. The first coordinate $\psi$ is a generalisation of the
classical Fatou coordinate in one dimension that was introduced in
\cite[Prop.~4.3]{BracciRaissyZaitsev2013Dynamicsofmultiresonantbiholomorphisms}
and examined more precisely in \cite[Proposition~3.1 and Lemma~3.3]{BracciRaissyStensonesAutomorphismsofmathbbCkwithaninvariantnonrecurrentattractingFatoucomponentbiholomorphictomathbbCtimesmathbbCk1},
where the following is shown:
\begin{prop}
\label{prop:PsiDefAndInj}For $F$ and $B_{h}$ as in Theorem~\ref{thm:BZ},
there exists a holomorphic map $\psi:B_{h}\to\mathbb{C}^{*}$, $\Re\psi>0$
such that
\begin{align}
\psi\circ F & =\psi+1,\label{eq:FatouEq}
\end{align}
and a constant $c\in\mathbb{C}$ depending only on $F_{{\rm N}}$
such that
\begin{align}
\psi(z) & =U+c\log(U)+O(U^{-1})\label{eq:FatouError}
\end{align}
for $z\in B_{h}$ and $U=(z^{1}\cdots z^{d})^{-k}$.

Moreover, there exists $R_{1}>\max\{R_{0},1\}$, $0<\theta_{1}<\theta_{0}$,
and $\beta_{0}<\beta_{1}<1/d$ such that the holomorphic map
\[
(\psi,\id):B_{h}(R_{1},\theta_{1},\beta_{1})\to(\mathbb{C}^{*})^{d},\quad z\mapsto(\psi(z),z^{2},\ldots,z^{d})
\]
is injective.
\end{prop}

The map $\psi$ is obtained as the uniform limit of the sequence $\{\psi_{n}\}_{n\in\mathbb{N}}$
of maps $B_{h}\to\mathbb{C}^{*}$ given by
\begin{align*}
\psi_{n}(z) & :=U_{n}-m+c\log(U_{n}),
\end{align*}
where $U_{n}=(z_{n}^{1}\cdots z_{n}^{d})^{-k}$ and $z_{n}=F^{n}(z)$
for $n\in\mathbb{N}$.
\begin{rem}
\label{rem:AsympPsi}In particular, $\psi(z)\sim U$ as $|U|\to\infty$
and by Proposition~\ref{prop:stableOrbitBehaviour} (and Remark~\ref{rem:stableOrbitBehaviourInBasins}),
$\psi(z_{n})\sim U_{n}$ as $n\to\infty$ uniformly in $z\in B_{h}$.
\end{rem}

\cite[Proposition~3.4]{BracciRaissyStensonesAutomorphismsofmathbbCkwithaninvariantnonrecurrentattractingFatoucomponentbiholomorphictomathbbCtimesmathbbCk1}
establishes further local coordinates to cover the remaining dimensions.
The following is a slight variation that will simplify the definition
of global coordinates in Section~\ref{subsec:Global-Basins}:
\begin{prop}
\label{prop:SigmaFatouExtra}Let $F$ and $B_{h}$ be as in Theorem~\ref{thm:BZ}
and $\psi:B_{h}\to\mathbb{C}^{*}$ as in Proposition~\ref{prop:PsiDefAndInj}.
For $j=2,\ldots,d$, there exists a holomorphic maps $\sigma_{j}:B_{h}\to\mathbb{C}^{*}$
such that
\begin{equation}
\sigma_{j}\circ F=\lambda_{j}\sigma_{j}\sqrt[kd]{\frac{\psi}{\psi+1}},\label{eq:SigmaEq}
\end{equation}
where the root is well-defined in the main branch since $\Re\psi>0$.
Moreover, for every $\alpha\in(1-\beta_{0},k)$ we have
\begin{equation}
\sigma_{j}(z)=z^{j}+O(u^{\alpha})\label{eq:SigmaError}
\end{equation}
for $z\in B_{h}$ and $u=z^{1}\cdots z^{d}$.
\end{prop}

\begin{rem}
\label{rem:AsympSigma}In particular for $j\ge2$, we have $\sigma_{j}(z_{n})\sim z_{n}^{j}$
as $n\to\infty$ uniformly for $z\in B_{h}$.
\end{rem}

\begin{proof}
For $2\le j\le d$ we will obtain $\sigma_{j}$ as the limit of the
sequence $\{\sigma_{j,n}\}_{n}$ of holomorphic maps $B_{h}\to\mathbb{C}^{*}$
defined for $n\in\mathbb{N}$ by
\[
\sigma_{j,n}(z)=\lambda_{j}^{-n}z_{n}^{j}\sqrt[kd]{\frac{\psi(z)+n}{\psi(z)}},
\]
where $(z_{n}^{1},\ldots,z_{n}^{d})=F^{n}(z)$ as usual. By Proposition~\ref{prop:stableOrbitBehaviour},
we have $z_{n}^{j}=O(n^{-1/kd})$ uniformly for $z\in B_{h}$, so
\begin{equation}
\sigma_{j,n}(z)=O(n^{-1/kd})\cdot\sqrt[kd]{1+\frac{n}{\psi(z)}}=O(1)\label{eq:sigmaUnifBound}
\end{equation}
uniformly for $z\in B_{h}$. To show convergence, we observe that
\begin{align*}
\sigma_{j,n+1}(z) & =\lambda_{j}^{-n-1}z_{n+1}^{j}\sqrt[kd]{\frac{\psi(z)+n+1}{\psi(z)}}\\
 & =\lambda_{j}^{-n-1}\paren[{\Big}]{\lambda_{j}{{\paren[{\Big}]{1-\frac{u_{n}^{k}}{kd}}}}+R_{j}(z_{n})}\sqrt[kd]{\frac{\psi(z)+n+1}{\psi(z)+n}}\sqrt[kd]{\frac{\psi(z)+n}{\psi(z)}}\\
 & =\sigma_{j,n}(z)\paren[{\Big}]{1-\frac{u_{n}^{k}}{kd}}\sqrt[kd]{\frac{\psi(z)+n+1}{\psi(z)+n}}+\lambda_{j}^{-n-1}R(z_{n})\sqrt[kd]{1+\frac{n+1}{\psi(z)}},
\end{align*}
where $R(z_{n})=O(\norm{z_{n}}^{l})=O(u_{n}^{\beta_{0}l})$, since
$z_{n}\in B_{h}$. Therefore with (\ref{eq:sigmaUnifBound}), we obtain
\begin{multline}
\sigma_{j,n+1}(z)-\sigma_{j,n}(z)\\
\begin{aligned} & =\sigma_{j,n}(z)\paren[{\Big}]{{{\paren[{\Big}]{1-\frac{u_{n}^{k}}{kd}}}}\sqrt[kd]{\frac{\psi(z)+n+1}{\psi(z)+n}}-1}+\lambda_{j}^{-n-1}R(z_{n})\sqrt[kd]{1+\frac{n+1}{\psi(z)}}\\
 & =O(1)\paren[{\Big}]{{{\paren[{\Big}]{1-\frac{u_{n}^{k}}{kd}}}}\sqrt[kd]{1+\frac{1}{\psi(z)+n}}-1}+O(u_{n}^{\beta_{0}l})O(n^{1/kd})
\end{aligned}
\label{eq:sigmaDiff}
\end{multline}
To estimate the first term on the right hand side, note that by (\ref{eq:FatouError})
we have
\[
\frac{1}{\psi(z)+n}=\frac{1}{\psi(z_{n})}=u_{n}^{k}\frac{1}{1+O(u_{n}^{k}\log(u_{n}^{k}))}=u_{n}^{k}+O(u_{n}^{2k}\log(u_{n}))
\]
and since $|u_{n}|=O(n^{-1/k})$ it follows

\begin{align*}
\paren[{\Big}]{1-\frac{u_{n}^{k}}{kd}}\sqrt[kd]{1+\frac{1}{\psi(z)+n}}-1 & =\paren[{\Big}]{1-\frac{u_{n}^{k}}{kd}}\sqrt[kd]{1+u_{n}^{k}+O(u_{n}^{2k}\log u_{n})}-1\\
 & =\paren[{\Big}]{1-\frac{u_{n}^{k}}{kd}}\paren[{\Big}]{1+\frac{u_{n}^{k}}{kd}+O(u_{n}^{2k}\log u_{n})}-1\\
 & =O(u_{n}^{2k}\log u_{n})\\
 & =|u_{n}|^{\alpha}O(n^{-2+\alpha/k}\log n)
\end{align*}
for any $\alpha\in(1/d,\min\{k,d\})$. Hence, again using $|u_{n}|=O(n^{-1/k})$,
(\ref{eq:sigmaDiff}) implies
\begin{equation}
\sigma_{j,n+1}(z)-\sigma_{j,n}(z)=|u_{n}|^{\alpha}O(n^{-2+\alpha/k}\log n+n^{-(\beta_{0}l-1/d-\alpha/k)/k}).\label{eq:SigmaDiffEstim}
\end{equation}
Since $2-\alpha/k>1$ and $\beta_{0}l-1/d-\alpha>2k-\alpha>k$, the
$O$-terms are summable, and for all $m\ge0$, we have $u_{n+m}=O(u_{n}),$
so summing up (\ref{eq:SigmaDiffEstim}), we obtain
\begin{equation}
\sigma_{j,n+m}(z)-\sigma_{j,n}(z)=O(u_{n}^{\alpha}).\label{eq:SigmaDiffVar}
\end{equation}
Since $\{u_{n}\}_{n}$ converges to $0$ uniformly as $n\to\infty$,
(\ref{eq:SigmaDiffVar}) implies that $\{\sigma_{j,n}\}_{n}$ converges
uniformly to a holomorphic map $\sigma_{j}:B_{h}\to\mathbb{C}$. For
$n=0$ and $m\to\infty$, (\ref{eq:SigmaDiffVar}) implies
\[
\sigma_{j}(z)-z^{j}=\sigma_{j}(z)-\sigma_{j,0}(z)=O(u^{\alpha}),
\]
showing (\ref{eq:SigmaError}).

It remains to show that $\sigma_{j}\neq0$. Since $\sigma_{j,n}\neq0$
for all $n\in\mathbb{N}$, Hurwitz's theorem implies that either $\sigma_{j}\equiv0$
or $\sigma_{j}(z)\neq0$ for all $z\in B_{h}$. For $r>0$ sufficiently
small, we have $(r,\ldots,r)\in B_{h}$ and, by (\ref{eq:SigmaError}),
we have
\[
\sigma_{j}(r,\ldots,r)=r+O(r^{d\alpha})=r(1+O(r^{d\alpha-1})).
\]
Since $\alpha>(d-1)/d$, this is non-zero for sufficiently small $r>0$,
showing that $\sigma_{j}\not\equiv0$.

Finally, for all $z\in B_{h}$ and $n\in\mathbb{N}$, we have
\begin{align*}
\sigma_{j,n}(F(z)) & =\lambda_{j}^{-n}z_{n+1}^{j}\sqrt[kd]{\frac{\psi(F(z))+n}{\psi(F(z))}}\\
 & =\lambda_{j}\underbrace{\paren[{\Big}]{\lambda_{j}^{-n-1}z_{n+1}^{j}\sqrt[kd]{\frac{\psi(z)+n+1}{\psi(z)}}}}_{=\sigma_{j,n+1}(z)}\sqrt[kd]{\frac{\psi(z)}{\psi(z)+1}},
\end{align*}
proving (\ref{eq:SigmaEq}).
\end{proof}
In the following, we will work in the variables $U=(z^{1}\cdots z^{d})^{-k}$
and $z'=(z^{2},\ldots,z^{d})$. Recalling the representation in Remark~\ref{rem:HomotopyType}
and noting that $u\to u^{-k}$ is injective on $S_{h}(R_{0},\theta_{0})$,
the variables $(U,z')$ still form a coordinate system on $B_{h}$
in which $B_{h}$ becomes
\begin{equation}
T(R_{0},\theta_{0},\beta_{0}):=\{(U,z')\mid U\in H(R_{0},\theta_{0}),|U|^{(\beta_{0}-1)/k}<|z^{2}\cdots z^{d}|,\norm{z'}_{\infty}<|U|^{-\beta_{0}/k}\}.\label{eq:BasinInUz}
\end{equation}
The next result, following \cite[Proposition~3.5]{BracciRaissyStensonesAutomorphismsofmathbbCkwithaninvariantnonrecurrentattractingFatoucomponentbiholomorphictomathbbCtimesmathbbCk1},
ensures that the maps $\psi,\sigma_{2},\ldots,\sigma_{d}$ still form
a coordinate system and their image contains a possibly smaller copy
of (\ref{eq:BasinInUz}).
\begin{prop}
\label{prop:FatouInjective}Let $F$ and $B_{h}$ be as in Theorem~\ref{thm:BZ}
and $\psi,\sigma_{2},\ldots,\sigma_{j}:B_{h}\to\mathbb{C}^{*}$ as
in Propositions~\ref{prop:PsiDefAndInj} and \ref{prop:SigmaFatouExtra}.
Then there exist $R_{1}>R_{0}$, $0<\theta_{1}<\theta_{0}$, and $\beta_{0}<\beta_{1}<1/d$
such that the holomorphic map
\[
\phi=(\psi,\sigma_{2},\ldots,\sigma_{d}):B_{h}(R_{1},\theta_{1},\beta_{1})\to(\mathbb{C^{*}})^{d}
\]
is injective. There further exist $R_{2}>0$, $\theta_{2}\in(0,\nicefrac{\pi}{2k})$,
and $\beta_{2}\in(0,1/d)$ such that
\begin{equation}
T(R_{2},\theta_{2},\beta_{2})\subseteq\phi(B_{h}).\label{eq:insideclaim}
\end{equation}
\end{prop}

\begin{proof}
Take $R_{1}>\max\{R_{0},1\}$, $0<\theta_{1}<\theta_{0}$, and $\beta_{0}<\beta_{1}<1/d$
from Proposition~\ref{prop:PsiDefAndInj}. Then for each $n\in\mathbb{N}$,
the map
\[
\phi_{n}=(\psi,\sigma_{2,n},\ldots,\sigma_{d,n}):B_{h}\to(\mathbb{C}^{*})^{d}
\]
 is injective on $B_{h}(R_{1},\theta_{1},\beta_{1})$. Hence, by
Hurwitz's theorem, the uniform limit $\phi=(\psi,\sigma_{2},\ldots,\sigma_{d})$
of the sequence $\{\phi_{n}\}_{n}$ is either injective or constant
on $B_{h}(R_{1},\theta_{1},\beta_{1})$.

As before, for $r>0$ sufficiently small, the point $(r,\ldots,r)$
lies in $B_{h}(R_{1},\theta_{1},\beta_{1})$. We will show that for
small values of $r>0$ the Jacobian of $\phi$ at $(r,\ldots,r)$
does not vanish. To simplify calculations, we work in coordinates
$(U,z')$ as above, so we compute the Jacobian of $\phi:T_{0}:=T(R_{0},\theta_{0},\beta_{0})\to(\mathbb{C}^{*})^{d}$
at
\[
w_{r}:=(U_{r},z'_{r}):=(r^{-kd},r,\ldots,r)\in T_{1}:=T(R_{1},\theta_{1},\beta_{1}).
\]
By Propositions~\ref{prop:PsiDefAndInj} and \ref{prop:SigmaFatouExtra},
we have
\[
\phi(U,z')=(U+c\log U+O(U^{-1}),z'+O(U^{-\alpha/k}))
\]
for $z\in B_{h}$.

Observe that since $R_{1}>\max\{R_{0},1\}$, $\theta_{1}<\theta_{0}$,
and $\beta_{1}>\beta_{0}$, we have
\begin{align*}
\delta_{0} & :=\min\{d(\partial H(R_{0},\theta_{0}),H(R_{1},\theta_{1}))/2,R_{1}-R_{1}^{\frac{1-\beta_{1}}{1-\beta_{0}}},R_{1}^{\beta_{1}/\beta_{0}}-R_{1}\}>0
\end{align*}
so for any $(U',z')\in T(R_{1},\theta_{1},\beta_{1})$ and $t\in\mathbb{R}$,
we have $U_{r}+\delta_{0}e^{it}\in H(R_{0},\theta_{0})$,
\[
|U+\delta_{0}e^{it}|^{(\beta_{0}-1)/k}\le(|U|-\delta_{0})^{(\beta_{0}-1)/k}<|U|^{(\beta_{1}-1)/k}<|z_{2}\cdots z_{d}|
\]
and
\[
\norm{z'}_{\infty}<|U|^{-\beta_{1}/k}<(|U|+\delta_{0})^{-\beta_{0}/k}\le|U+\delta_{0}e^{it}|^{-\beta_{0}/k},
\]
implying $(U+\delta_{0}e^{it},z')\in T(R_{0},\theta_{0},\beta_{0})$.
In particular for $r>0$ such that $w_{r}\in T(R_{1},\theta_{1},\beta_{1})$,
and $t\in\mathbb{R}$ we have
\begin{equation}
(U_{r}+\delta_{0}e^{it},z_{r}')\in T(R_{0},\theta_{0},\beta_{0}),\label{eq:UrNbh}
\end{equation}
and for all $j\ge2$ we have
\begin{equation}
(w_{r}+r\delta_{1}e^{it}e_{j})\in T(R_{0},\theta_{0},\beta_{0}),\label{eq:zrNbh}
\end{equation}
where $\delta_{1}=1-R_{1}^{-(1/d-\beta_{0})/k}>0$. Let $h:B_{h}\to\mathbb{C}$
with $h(U,z')=O(U^{-\nu})$ for $(U,z')\in T_{0}$ for some $\nu>0$.
Then there exists $g:B_{h}\to\mathbb{C}$ such that $g(U,z')=O(1)$
and $h(U,z')=U^{-\nu}g(U,z')$ for $(U,z')\in T_{0}$. For $w_{r}\in T_{1}$,
by (\ref{eq:UrNbh}), we then have
\[
\abs[{\bigg}]{\frac{\partial h}{\partial U}(w_{r})}=\frac{1}{2\pi}\abs[{\bigg}]{\int_{|U-U_{r}|=\delta_{0}}\frac{U^{-\nu}g(U,z_{r}')}{(U-U_{r})^{2}}\,d\zeta}\le\frac{1}{\delta_{0}}(r^{-kd}+\delta_{0})^{-\nu}\sup_{(U,z_{r}')\in T_{0}}|g(U,z_{r}')|=O(r^{\nu kd}),
\]
and by (\ref{eq:zrNbh}), we have
\[
\abs[{\bigg}]{\frac{\partial h}{\partial z^{j}}(w_{r})}=\frac{1}{2\pi}\abs[{\bigg}]{\int_{|\zeta|=\delta_{1}r}\frac{r^{\nu kd}g(x_{r}+\zeta e_{j})}{\zeta^{2}}\,d\zeta}\le\frac{1}{\delta_{1}r}r^{\nu kd}\sup_{w\in T_{0}}|g(w)|=O(r^{\nu kd-1}).
\]
Hence, for all $i,j\ge2$, we obtain
\begin{alignat*}{2}
\frac{\partial\psi}{\partial U}(w_{r}) & =1+O(r^{kd}),\qquad & \frac{\partial\psi}{\partial z^{j}}(w_{r}) & =O(r^{kd-1}),\\
\frac{\partial\sigma_{i}}{\partial U}(w_{r}) & =O(r^{\alpha d}), & \frac{\partial\sigma_{i}}{\partial z^{j}}(w_{r}) & =\delta_{ij}+O(r^{\alpha d-1}).
\end{alignat*}
So for the products in the Jacobian firstly we have
\[
\prod_{j=1}^{d}\frac{\partial\phi_{j}}{\partial x^{j}}(x_{r})=1+O(r^{\alpha d-1})
\]
and secondly for every $\rho\in S_{d}\backslash\{\id\}$, there exists
$j\le d$ such that $\rho(j)\neq j$ and hence $\frac{\partial\phi_{j}}{\partial x^{\rho(j)}}=O(r^{\alpha d-1})$,
so we have
\[
\prod_{j=1}^{d}\frac{\partial\phi_{j}}{\partial x^{\rho(j)}}(x_{r})=O(r^{\alpha d-1}).
\]
In conclusion the Leibniz formula yields:
\begin{align*}
\Jac_{(r^{-kd},r,\ldots,r)}\phi & =1+O(r^{\alpha d-1})+\sum_{\rho\in S_{d}\backslash\{\id\}}O(r^{\alpha d-1})=1+O(r^{\alpha d-1}),
\end{align*}
and since $\alpha d>d-1$, this is non-zero for sufficiently small
$r>0$, showing that $\phi$ is injective on $B(R_{1},\theta_{1},\beta_{1})$.

Now let $R_{1}'>R_{1}$, $0<\theta_{1}'<\theta_{1}$, $\beta_{1}<\beta_{1}'<1/d$,
and $T_{1}':=T(R_{1}',\theta_{1}',\beta_{1}')$ then the closure $\overline{T_{1}'}$
is contained in $T_{1}$. To show (\ref{eq:insideclaim}), we show
that there exist $R_{2}>1$, $\theta_{2}\in(0,\nicefrac{\pi}{2k})$,
and $\beta_{2}\in(0,1/d)$, such that
\begin{equation}
\phi(\partial T_{1}')\cap T(R_{2},\theta_{2},\beta_{2})=\emptyset\quad\text{and}\quad\phi(T_{1}')\cap T(R_{2},\theta_{2},\beta_{2})\neq\emptyset.\label{eq:insideNotHitBdr}
\end{equation}
Since $\phi$ is an embedding of a neighbourhood of $\overline{T_{1}'}$,
we have $\partial\phi(T_{1}')=\phi(\partial T_{1}')$, and, since
$T_{2}$ is connected, (\ref{eq:insideNotHitBdr}) implies $T_{2}\subseteq\phi(T_{1}')$.

Fix $0<\theta_{2}<\theta_{1}$ and $\beta_{1}<\beta_{2}<1/d$  and
let $w=(U,z')\in\partial T^{1}$. We have three cases:
\begin{casenv}
\item If $U\in\partial H(R_{1}',\theta_{1}')$, there exists $C>0$ such
that
\[
d(U,H(R_{2},\theta_{2}))>C|U|,
\]
for all $|U|>R_{1}+1$ and every $R_{2}>0$, and by (\ref{eq:FatouError})
we have
\[
|\psi(z)-U|=o(U),
\]
so for $|U|$ large enough, we have
\begin{align}
d(\phi(w),T(R,\theta_{2},\beta_{2})) & \ge d(\psi(w),H(R_{2},\theta_{2}))\nonumber \\
 & \ge d(U,H(R_{2},\theta_{2}))-|\psi(z)-U|\label{eq:bdDistU}\\
 & >0,\nonumber
\end{align}
for any $R>0$.
\item If $|U|^{(\beta_{1}'-1)/k}=|z^{2}\cdots z^{d}|$, then from (\ref{eq:SigmaError})
and Remark~\ref{rem:AsympPsi} it follows
\begin{align*}
|\sigma_{2}(w)\cdots\sigma_{d}(w)| & =|U|^{(\beta_{1}'-1)/k}+O(|U|^{-\alpha/k})\\
 & \approx|U|^{(\beta_{1}'-1)/k}\\
 & \approx|\psi(w)|^{(\beta_{2}-1)/k}|U|^{-(\beta_{2}-\beta_{1}')/k},
\end{align*}
so for $|U|$ large enough $|\sigma_{2}(w)\cdots\sigma_{d}(w)|<|\psi(w)|^{(\beta_{2}-1)/k}$.
\item If $|z^{j}|=|U|^{-\beta_{1}'/k}$, then by (\ref{eq:SigmaError})
and Remark~\ref{rem:AsympPsi} we have
\begin{align*}
|\sigma_{j}(w)| & =|U|^{-\beta_{1}'/k}-O(|U|^{-\alpha/k})\\
 & \approx|U|^{-\beta_{1}'/k}\\
 & \approx|\psi(z)|^{-\beta_{2}/k}|U|^{(\beta_{2}-\beta_{1}')/k},
\end{align*}
so for $|U|$ large enough, $|\sigma_{j}(w)|>|\psi(z)|^{-\beta_{2}/k}$.
\end{casenv}
In conclusion, there exists $R_{3}>0$ such that $\phi(U,z')\notin T(R,\theta_{2},\beta_{2})$
for every $R>0$ and $(U,z')\in\partial T_{1}'$ such that $|U|>R_{3}$.
Since $\Re\psi(z)\approx\Re U\approx|U|$, we can take $R_{2}$ large
enough that $\Re\psi(z)<R_{2}$ whenever $|U|\le R_{3}$, so $\phi(w)\notin T(R_{2},\theta_{2},\beta_{2})$
for all $w\in\partial T_{1}'$.

Let again $w_{r}:=(U_{r},z_{r}'):=(r^{-kd},r,\ldots,r)$ for $r>0$.
As in (\ref{eq:bdDistU}), by (\ref{eq:FatouError}) we have
\begin{align*}
d(r^{-kd},\partial H(R_{2},\theta_{2})) & \approx r^{-kd}\quad\text{and}\quad|\psi(w_{r})-r^{-kd}|=o(r^{-kd}),
\end{align*}
hence for $r^{-kd}>R_{2}$ large enough,
\[
d(r^{-kd},\partial H(R_{2},\theta_{2}))>|\psi(w_{r})-r^{-kd}|
\]
and $\psi(w_{r})\in H(R_{2},\theta_{2})$. Again by (\ref{eq:FatouError})
and (\ref{eq:SigmaError}), for small $r>0$, we have
\begin{align*}
|\sigma_{2}(w_{r})\cdots\sigma_{d}(w_{r})| & =r^{d-1}+O(r^{\alpha d})\approx r^{d-1}, & |\psi(w_{r})|^{-\beta_{2}/k} & \approx r^{\beta_{2}d},\\
|\psi(w_{r})|^{(\beta_{2}-1)/k} & \approx r^{d-d\beta_{2}}=o(r^{d-1}), & |\sigma_{j}(w_{r})|=r+O(r^{\alpha d})\approx r & =o(r^{\beta_{2}d}).
\end{align*}
Hence $\phi(w_{r})\in T(R_{2},\theta_{2},\beta_{2})$ for $r>0$ small
enough and we have shown (\ref{eq:insideNotHitBdr}).
\end{proof}

\subsection{\label{subsec:InternalOrbitBehaviour}Orbit behaviour}

The coordinates from the previous section give us some more precise
information about the dynamics in $B_{h}$.
\begin{notation}
For $z=(z^{1},\ldots,z^{d})\in\mathbb{C}^{d}$, denote $|z|:=(|z^{1}|,\ldots,|z^{d}|)$.
\end{notation}

\begin{prop}
\label{prop:DirectionsOfConvergence}Let $F$ and $B_{h}$ be as in
Theorem~\ref{thm:BZ}. Then for $z\in B_{h}$ and $z_{n}=F^{n}(z)$
for $n\in\mathbb{N}$, the limit
\begin{equation}
v(z):=\lim_{n\to\infty}|z_{n}|/\norm{z_{n}}_{2}\in\mathbb{R}_{+}^{d}\label{eq:ModDirectionConvergence}
\end{equation}
is a unit vector with positive entries and the set of accumulation
points of the sequence of directions $\{z_{n}/\norm{z_{n}}_{2}\}_{n}$
is
\begin{equation}
\omega(z):=\{w\in\mathbb{C}^{d}\mid|w|=v(z),\arg w^{1}+\cdots+\arg w^{d}\equiv2\pi h/k\}.\label{eq:LimitSets}
\end{equation}
In other words, $\{z_{n}\}_{n}$ converges to $0$ tangent to the
linear cone $\mathbb{R}_{0}^{+}\cdot\omega(z)$.

Furthermore, for any positive unit vector $v\in\mathbb{R}_{+}^{d}$,
there exists $z\in B_{h}$ such that $v(z)=v$. Hence the set of all
accumulation points of directions of orbits in $B_{h}$ is
\[
\omega(B_{h})=\{w\in(\mathbb{C}^{*})^{d}\mid\arg w^{1}+\cdots+\arg w^{d}\equiv2\pi h/k\}.
\]
\end{prop}

\begin{proof}
Let $\psi$ and $\sigma_{2},\ldots,\sigma_{d}$ be as in Propositions~\ref{prop:PsiDefAndInj}
and \ref{prop:SigmaFatouExtra} and set
\[
\sigma_{1}:=e^{2\pi ih/k}(\sqrt[k]{\psi}\sigma_{2}\cdots\sigma_{d})^{-1}:B_{h}\to\mathbb{C}^{*},
\]
where the root is well-defined, since $\Re\psi>0$ and we always choose
its values in the main branch. However, since $u_{n}=z_{n}^{1}\cdots z_{n}^{d}\in S_{h}(R_{0},\theta_{0})$
is near the direction $e^{2\pi ih/k}$ for $z\in B_{h}$, we have
$u_{n}=e^{2\pi ih/k}U^{-1/k}$, hence
\[
\sigma_{1}(z_{n})\sim e^{2\pi ih/k}(\sqrt[k]{U_{n}}z_{n}^{2}\cdots z_{n}^{d})^{-1}=u_{n}/(z_{n}^{2}\cdots z_{n}^{d})=z_{n}^{1}
\]
as $n\to\infty$. Moreover, the functional equation (\ref{eq:SigmaEq})
for $\sigma_{j},j\ge2$ implies the same for $\sigma_{1}$:
\begin{align*}
\sigma_{1}\circ F & =e^{2\pi ih/k}((\psi+1)^{\expfrac 1k-\expfrac{(d-1)}{kd}}\psi^{\expfrac{(d-1)}{kd}}\lambda_{2}\sigma_{2}\cdots\lambda_{d}\sigma_{d})^{-1}\\
 & =\lambda_{1}e^{2\pi ih/k}(\sqrt[k]{\psi}\sigma_{2}\cdots\sigma_{d})^{-1}\psi^{\expfrac 1{kd}}(\psi+1)^{-\expfrac 1{kd}}\\
 & =\lambda_{1}\sigma_{1}\sqrt[kd]{\frac{\psi}{\psi+1}}.
\end{align*}
 Hence, for $\sigma=(\sigma_{1},\ldots,\sigma_{d})$, we have
\begin{equation}
z_{n}/\norm{z_{n}}_{2}\sim\sigma(z_{n})/\norm{\sigma(z_{n})}_{2}\sim\Lambda^{n}\sqrt[kd]{\psi}\sigma(z)/\norm{\sqrt[kd]{\psi}\sigma(z)}_{2}\label{eq:directionsAsymp}
\end{equation}
as $n\to+\infty$, where $\Lambda=\diag(\lambda_{1},\ldots,\lambda_{d})$,
so the limit in (\ref{eq:ModDirectionConvergence}) exists and is
equal to
\[
v(z)=|\sigma(z)|/\norm{\sigma(z)}_{2}\in\mathbb{R}_{+}^{d}.
\]
Since $\arg(\sqrt[k]{\psi})+\arg(\sigma_{1})+\cdots+\arg(\sigma_{d})\equiv2\pi h/k$
by definition, $\sqrt[kd]{\psi}\sigma(z)/\norm{\sqrt[kd]{\psi}\sigma(z)}$
lies in $\omega(z)$ and since (\ref{eq:LimitSets}) is invariant
under multiplication by $\Lambda$, so does every accumulation point
of $\{z_{n}/\norm{z_{n}}_{2}\}_{n}$.

To show that all points in (\ref{eq:LimitSets}) occur, note that
one-resonance implies that the angles $\arg(\lambda_{2}),\ldots,\arg(\lambda_{d})$
are rationally independent modulo $2\pi$, and this implies, e.g.\ by
\cite[Corollary~I.7]{Zehnder2010Lecturesondynamicalsystems}, that
the sequence $\{(\lambda_{2}^{n},\ldots,\lambda_{d}^{n})\}_{n}$ is
dense in $(S^{1})^{d-1}$. Hence (\ref{eq:directionsAsymp}) shows
that $\{z_{n}/\norm z_{2}\}_{n}$ accumulates on the whole set (\ref{eq:LimitSets}).

Finally, let $v=(v^{1},\ldots,v^{d})\in\mathbb{R}^{+}$ and $R_{2},\theta_{2},\beta_{2}$
as in Proposition~\ref{prop:FatouInjective}, so $B_{h}(R_{2},\theta_{2},\beta_{2})\subseteq\sigma(B_{h})$.
For $\varepsilon>0$, let $v_{\varepsilon}:=\varepsilon e^{2\pi i\expfrac h{kd}}v$.
Then for $\varepsilon>0$ small enough, we have $\pi(v_{\varepsilon})=\varepsilon^{d}v^{1}\cdots v^{d}e^{2\pi ih/k}\in S_{h}(R_{2},\theta_{2})$
and

\begin{align*}
|v_{\varepsilon}^{j}|=\varepsilon|v^{j}| & <\varepsilon^{d\beta_{2}}|v^{1}\cdots v^{d}|^{\beta_{2}}=|\pi(v_{\varepsilon})|^{\beta_{2}}\quad\text{for }1\le j\le d,
\end{align*}
since $d\beta_{2}<1$, and hence $v_{\varepsilon}\in B_{h}(R_{2},\theta_{2},\beta_{2})\subseteq\sigma(B_{h})$,
i.e.\ there exists $z\in B_{h}$ such that $\sigma(z)=v_{\varepsilon}$
and
\[
v(z)=|\sigma(z)|/\norm{\sigma(z)}_{2}=\varepsilon v/\varepsilon=v.\qedhere
\]
\end{proof}
\begin{rem}[$d=2$]
\label{rem:Orbitsd2}For $d=2$ and $(z_{0},w_{0})\in B_{h}$, let
$m=v(z_{0},w_{0})\in(0,+\infty)$. Then the linear cone $\mathbb{R}_{0}^{+}\cdot\omega(z_{0},w_{0})$
is in fact a real $2$-dimensional linear subspace of $\mathbb{C}^{2}$
given by $\{z=me^{2\pi ih/k}\overline{w}\}$. The complex lines intersecting
this subspace are precisely those of the form $\{z=me^{it}w\}$ for
$t\in\mathbb{R}$ and all intersections are transversal, so it is
not contained in any proper complex subspace of $\mathbb{C}^{2}$.

Recall the representation in Figure~\ref{fig:Argument-components}
of $B_{h}$ in polar decomposition from Remark~\ref{rem:externalGeom}.
In polar coordinates $(z,w)=(r_{1}e^{is},r_{2}e^{it})$, the linear
cone $\mathbb{R}_{+}\cdot\omega(z_{0},w_{0})$ has the form
\begin{equation}
\{(r_{1},r_{2})\in\mathbb{R}_{+}^{2}\mid r_{1}=mr_{2}\}\times\{s+t=2\pi h/k\}.\label{eq:convRsubspace}
\end{equation}
Hence Proposition~\ref{prop:DirectionsOfConvergence} translates
to the fact that the modulus component converges to $0$ tangential
to the line $|z|=m|w|$ and the argument component accumulates on
the whole central curve $s+t\equiv2\pi h/k$. Moreover, each value
$m\in(0,+\infty)$ occurs, consistent with the fact that the modulus
component contains lines with any possible slope $m\in(0,+\infty)$.
\end{rem}

\begin{rem}[$d>2$]
\label{rem:orbitsdbig}For $d>2$ and $z\in B_{h}$, the punctured
linear cone $\mathbb{R}_{+}\cdot\omega(z)$ still forms a real $d$-dimensional
submanifold of $\mathbb{C}^{d}$, but its closure $\mathbb{R}_{0}^{+}\cdot\omega(z)$,
has a singularity at $0$. In fact $\mathbb{R}_{+}\cdot\omega(z)$
is not even contained in any proper real subspace of $\mathbb{C}^{d}$.
\end{rem}

\begin{proof}
Let $(v^{1},\ldots,v^{d})=v(z)$ and $\zeta$ a primitive $d-1$-st
root of unity. Then any real subspace containing $\mathbb{R}_{+}\cdot\omega(z)$
already contains
\begin{align*}
\frac{1}{d-1}\sum_{m=1}^{d-1}(1,\zeta^{m},\ldots,\zeta^{m},\zeta^{m}e^{2\pi ih/k}) & =e_{1},\\
\frac{1}{d-1}\sum_{j=1}^{d-1}(i,\zeta^{m},\ldots,\zeta^{m},-i\zeta^{m}e^{2\pi ih/k}) & =ie_{1},
\end{align*}
and similarly $e_{j}$ and $ie_{j}$ for $j=2,\ldots,d$, hence it
has to be $\mathbb{C}^{d}$.
\end{proof}
In particular, the above remarks imply:
\begin{cor}
\label{cor:NoComplexTangentOrbits}No orbit of $F$ inside the basins
$B_{0},\ldots,B_{k-1}$ as in Theorem~\ref{thm:BZ} converges to
$0$ tangent to a proper complex subspace of $\mathbb{C}^{d}$.
\end{cor}

\begin{proof}
Assume $z\in B_{h}$ and $\{z_{n}\}_{n}$ converges to $0$ tangent
to a complex subspace $V\subseteq\mathbb{C}^{d}$. Then by Proposition~\ref{prop:DirectionsOfConvergence}
$V$ has to contain the linear cone $\mathbb{R}^{+}\cdot\omega(z)$.
Thus Remarks~\ref{rem:Orbitsd2} and \ref{rem:orbitsdbig} imply
that $V=\mathbb{C}^{d}$.
\end{proof}
Remark~\ref{rem:OrbitsIntro} now follows from Proposition~\ref{prop:DirectionsOfConvergence}
and Corollary~\ref{cor:NoComplexTangentOrbits}.

\subsection{\label{subsec:Global-Basins}Geometry of the global basins}

By jet-interpolation, we may choose $F$ to be a global automorphism
of $\mathbb{C}^{d}$. We then use a variant of the coordinates on
each local basin from Section~\ref{subsec:Fatou-coordinates} that
extends to a biholomorphism from the corresponding global basin to
$\mathbb{C}\times(\mathbb{C}^{*})^{d-1}$.

We use the following result from \cite{Weickert1998AttractingbasinsforautomorphismsofbfC2}
and \cite[Corollary~2.2]{Forstneriv1999InterpolationbyholomorphicautomorphismsandembeddingsinmathbbCn}:
\begin{thm}
\label{thm:JetInterp}For every invertible germ of endomorphisms $F_{0}$
of $\mathbb{C}^{d}$ at the origin and every $l\in\mathbb{N}$, there
exists an automorphism $F\in\Aut(\mathbb{C}^{d})$ such that $F(z)=F_{0}(z)+O\paren{{{\norm z}}^{l}}.$
\end{thm}

For $F_{0}=F_{{\rm N}}$ and $l\in\mathbb{N}$ as in Theorem~\ref{thm:BZ},
this implies that there exist biholomorphisms $F$ of $\mathbb{C}^{d}$
of the form
\begin{equation}
F(z)=F_{{\rm N}}(z)+O\paren{{{\norm z}}^{l}}\label{eq:BiholWithTail}
\end{equation}
 with local attracting basins $B_{0},\ldots,B_{k-1}$.
\begin{rem}
For $F\in\Aut(\mathbb{C}^{d})$ of the form (\ref{eq:BiholWithTail}),
the global basins $\Omega_{0},\ldots,\Omega_{k-1}$ are growing unions
of biholomorphic preimages of $B_{0},\ldots,B_{k-1}$. As such they
are still pairwise disjoint and open, invariant and attracted to $0$
under $F$, and homotopy equivalent to $(S^{1})^{d-1}$.
\end{rem}

To show that these global basins are in fact biholomorphic to $\mathbb{C}\times(\mathbb{C}^{*})^{d-1}$,
we wish to extend the coordinates from the previous section to the
global basins via their functional equations (\ref{eq:FatouEq}) and
(\ref{eq:SigmaEq}). However, the equation (\ref{eq:SigmaEq}) involves
division by $\psi+1$, which has zeros. In \cite{BracciRaissyStensonesAutomorphismsofmathbbCkwithaninvariantnonrecurrentattractingFatoucomponentbiholomorphictomathbbCtimesmathbbCk1}
this problem is circumvented by restricting to an exhausting sequence
of subsets of $\Omega_{h}$ and constructing a fibre bundle biholomorphic
to $\Omega_{h}$ with total space $\mathbb{C}\times(\mathbb{C}^{*})^{d-1}$.
We will instead replace $\sigma_{j}$ by a coordinate with a simpler
functional equation, that allows for global extension (compare \cite{Reppekus2019PuncturednonrecurrentSiegelcylindersinautomorphismsofmathbbC2}):
\begin{cor}
\label{cor:FatouTauExtra}Assume the setting of Proposition~\ref{prop:FatouInjective}.
For $2\le j\le d$, the map $\tau_{j}=\sqrt[kd]{\psi}\sigma_{j}:B_{h}\to\mathbb{C}^{*}$
is well-defined and satisfies
\begin{align}
\tau_{j}\circ F & =\lambda_{j}\tau_{j}.\label{eq:tauEq}
\end{align}
Moreover, the map $(\psi,\tau_{2},\ldots,\tau_{d})$ is injective
on $B_{h}(R_{1},\theta_{1},\beta_{1})$ and its image contains the
set
\begin{equation}
\{(U,w')\in H(R_{2},\theta_{2})\times\mathbb{C}^{d-1}\mid|U|^{(\beta_{2}-1/d)/k}<|w_{2}\cdots w_{d}|,\norm{w'}_{\infty}<|U|^{(1/d-\beta_{2})/k}\}.\label{eq:insideclaimTau}
\end{equation}
\end{cor}

\begin{proof}
Fix $2\le j\le d$. Since $\Re\psi>0$, the root $\sqrt[kd]{\psi}$
is well-defined. (\ref{eq:tauEq}) follows directly from (\ref{eq:SigmaEq}):
\begin{align*}
\tau_{j}\circ F & =\sqrt[kd]{\psi\circ F}\cdot\sigma_{j}\circ F=\sqrt[kd]{\psi+1}\lambda_{j}\sigma_{j}\sqrt[kd]{\frac{\psi}{\psi+1}}=\lambda_{j}\tau_{j}.
\end{align*}
Injectivity of $(\psi,\tau_{2},\ldots,\tau_{d})$ and (\ref{eq:insideclaimTau})
follow from Proposition~\ref{prop:FatouInjective}, since $(\zeta,\xi)\mapsto(\zeta,\sqrt[kd]{\zeta}\xi)$
is well-defined and injective for $\Re\zeta>0$ and $\xi\in\mathbb{C}^{d-1}$,
and (\ref{eq:insideclaimTau}) is the image of $T(R_{2},\theta_{2},\beta_{2})$
under that map.
\end{proof}
Now if $F$ is an automorphism, this new system of coordinates extends
indefinitely:
\begin{prop}
\label{prop:FatouBihol}Let $F$ be an automorphism of the form (\ref{eq:BiholWithTail}),
$B_{h}$ as in Theorem~\ref{thm:BZ}, $\Omega_{h}=\bigcup_{n}F^{-n}(B_{h})$,
and $\psi,\tau_{2},\ldots,\tau_{d}:B_{h}\to\mathbb{C}^{*}$ as in
Proposition~\ref{prop:PsiDefAndInj} and Corollary~\ref{cor:FatouTauExtra}.
Let $\hat{\psi}:\Omega_{h}\to\mathbb{C}$ and $\hat{\tau}_{2},\ldots,\hat{\tau}_{d}:\Omega_{h}\to\mathbb{C}^{*}$
be given by
\[
\hat{\psi}(z)=\psi(F^{n}(z))-n
\]
and
\[
\hat{\tau}_{j}(z)=\lambda_{j}^{-n}\tau_{j}(F^{n}(z))\quad\text{for }2\le j\le d
\]
for $z\in F^{-n}(z)$ and $n\in\mathbb{N}$. Then
\[
\hat{\phi}=(\hat{\psi},\hat{\tau}_{2},\ldots,\hat{\tau}_{d}):\Omega_{h}\to\mathbb{C}\times(\mathbb{C}^{*})^{d-1}
\]
is a well-defined biholomorphism. In particular $\Omega_{h}$ is biholomorphic
to $\mathbb{C}\times(\mathbb{C}^{*})^{d-1}$.
\end{prop}

\begin{proof}
Let $m>n$ such that $z\in F^{-n}(z)\subseteq F^{-m}(z)$. Then
\[
\psi(F^{m}(z))-m=\psi(F^{m-n}(F^{n}(z)))-m=\psi(F^{n}(z))-n
\]
and
\[
\lambda_{j}^{-m}\sigma_{j}(F^{m}(z))=\lambda_{j}^{-m}\sigma_{j}(F^{m-n}(F^{n}(z)))=\lambda_{j}^{-n}\sigma_{j}(F^{n}(z)),
\]
so $\hat{\phi}$ is well-defined.

For injectivity, let $z,w\in\Omega_{h}$. Then by Part~(\ref{enu:OB3Bh})
of Proposition~\ref{prop:stableOrbitBehaviour} (and Remark~\ref{rem:stableOrbitBehaviourInBasins})
there exists $n\in\mathbb{N}$ such that $F^{n}(z),F^{n}(w)\in B_{h}(R_{1},\theta_{1},\beta_{1})$.
Now $\hat{\phi}(z)=\hat{\phi}(w)$ implies
\[
\phi(F^{n}(z))=\phi(F^{n}(w)),
\]
and by injectivity of $\phi$ on $B_{h}(R_{1},\theta_{1},\beta_{1})$
and of $F$ on $\mathbb{C}^{d}$, we have $F^{n}(z)=F^{n}(w)$ and
$z=w$, showing that $\hat{\phi}$ is injective.

To show surjectivity, let $(\zeta,\xi')\in\mathbb{C}\times(\mathbb{C}^{*})^{d-1}$.
Then for $n\in\mathbb{N}$ large enough, we have $\zeta+n\in H(R_{2},\theta_{2})$,
\[
|\zeta+n|^{-(\beta_{2}-1/d)/k}<|\xi^{2}\cdots\xi^{d}|\quad\text{and}\quad\norm{\xi'}_{\infty}<|\zeta+n|^{(1/d-\beta_{2})/k},
\]
since $\beta_{2}<1/d$. Hence by (\ref{eq:insideclaimTau}),
\[
(\zeta+n,(\Lambda')^{n}\xi')\in\hat{\phi}(B_{h}),
\]
where $\Lambda':=\diag(\lambda_{2},\cdots,\lambda_{d})$, so there
exists $z\in B_{h}$ such that $\hat{\phi}(z)=(\zeta+n,(\Lambda')^{n}\xi')$
and
\[
\hat{\phi}(F^{-n}(z))=(\zeta,\xi'),
\]
showing surjectivity.
\end{proof}
The second part of Theorem~\ref{thm:FatouCCstarMulti} now follows
from Theorem~\ref{thm:JetInterp} and the following corollary to
Proposition~\ref{prop:FatouBihol}:
\begin{cor}
\label{cor:CylinderCoord}Let $F$ be an automorphism of the form
(\ref{eq:BiholWithTail}), $B_{h}$ as in Theorem~\ref{thm:BZ},
and $\Omega_{h}=\bigcup_{n}F^{-n}(B_{h})$. There exists a biholomorphic
map $\phi_{h}:\Omega_{h}\to\mathbb{C}\times(\mathbb{C}^{*})^{d-1}$
conjugating $F$ to
\begin{equation}
(\zeta,\xi)\mapsto(\zeta+1,\xi).\label{eq:cylCoord}
\end{equation}
\end{cor}

\begin{proof}
The biholomorphic map $\hat{\phi}$ from Proposition~\ref{prop:FatouBihol}
conjugates $F$ to
\begin{equation}
(\zeta,\xi)\mapsto(\zeta+1,\Lambda'\xi),\label{eq:cylinderRot}
\end{equation}
where $\Lambda':=\diag(\lambda_{2},\cdots,\lambda_{d})$. The map
\[
\eta:\mathbb{C}\times(\mathbb{C}^{*})^{d-1}\to\mathbb{C}\times(\mathbb{C}^{*})^{d-1},\quad\eta(\zeta,\xi)=(\zeta,(\Lambda')^{-\zeta}\xi)
\]
is biholomorphic and well-defined up to choice of a logarithm of the
invertible matrix $\Lambda'$ and further conjugates (\ref{eq:cylinderRot})
to (\ref{eq:cylCoord}), so $\phi_{h}:=\eta\circ\hat{\phi}$ has the
required properties.
\end{proof}

\section{\label{sec:Periodic-cycles}Periodic cycles}

In this section we prove Theorem~\ref{thm:FatouCCstarPeriodic} and
Corollary~\ref{cor:CCstarVariants} via explicit construction. We
first show the existence of ``roots up to order $l\in\mathbb{N}$''
for one-resonant germs:
\begin{lem}
\label{lem:RootGerm}Let $d,k\in\mathbb{N}^{*}$ and $F_{0}\in\Aut(\mathbb{C}^{d},0)$
be be one-resonant of index $\alpha$ of the form
\[
F_{0}(z)=\Lambda z(1+cz^{k\alpha}),
\]
where $c\in\mathbb{C}\backslash\{0\}$ and $\Lambda=\diag(\lambda_{1},\ldots,\lambda_{d})$.
Then for every $p\in\mathbb{N}^{*}$ dividing $k$ and $l\in\mathbb{N}$
there exists a germ $F_{p}\in\Aut(\mathbb{C}^{d},0)$, one resonant
of index $p\alpha$, of the form
\[
F_{p}(z)=M_{p}z\paren{1+c/pz^{k\alpha}}+O(\norm z^{2k|\alpha|+1}),
\]
where $M_{p}=\diag(\mu_{1},\ldots,\mu_{d})$ is such that $M_{p}^{p}=\Lambda$
and $\mu_{1}\cdots\mu_{d}=\zeta_{p}:=e^{2\pi i/p}$, such that for
all germs $F$ such that $F(z)=F_{p}(z)+O(\norm z^{l})$, the $p$-th
iterate $F^{p}$ has the form
\[
F^{p}(z)=F_{0}(z)+O(\norm z^{l}).
\]
\end{lem}

\begin{proof}
We first determine the iterates of the general germ
\[
F_{1}(z)=Mz(1+az^{k\alpha}+bz^{2k\alpha})
\]
with $a,b\in\mathbb{C}$ and $M$ diagonal such that $M^{k\alpha}=1$.
Then for every $m\in\mathbb{N}$, the $m$-th iterate has the form
\[
F_{1}^{m}(z)=M^{m}z(1+a_{m}z^{k\alpha}+b_{m}z^{2k\alpha}+O(z^{3k\alpha})),
\]
and we have
\begin{align*}
(F_{1}^{m}(z))^{k\alpha} & =z^{k\alpha}+a_{m}k|\alpha|z^{2k\alpha}+O(z^{3k\alpha}),\\
(F_{1}^{m}(z))^{2k\alpha} & =z^{2k\alpha}+O(z^{3k\alpha}),
\end{align*}
so
\begin{align*}
F_{1}\circ F_{1}^{m}(z) & =M^{m+1}z(1+a_{m}z^{k\alpha}+b_{m}z^{2k\alpha})(1+az^{k\alpha}+(aa_{m}k|\alpha|+b)z^{2k\alpha}+O(z^{3k\alpha}))\\
 & =M^{m+1}z(1+(a_{m}+a)z^{k\alpha}+(b_{m}+b+a_{m}a(k|\alpha|+1))z^{2k\alpha}).
\end{align*}
From this, we obtain and solve recursive expressions for $m\in\mathbb{N}$:
\begin{align*}
a_{m} & =a_{m-1}+a=ma\quad\text{and }\\
b_{m} & =b_{m-1}+b+(m-1)a^{2}(k|\alpha|+1)\\
 & =mb+\sum_{j=1}^{m-1}ja^{2}(k|\alpha|+1)\\
 & =mb+\frac{m(m-1)}{3}a^{2}(k|\alpha|+1).
\end{align*}
So in particular
\[
F_{1}^{p}(z)=M^{p}z\paren[{\Big}]{1+paz^{k\alpha}+p{{\paren[{\Big}]{b+\frac{p-1}{2}a^{2}(k|\alpha|+1)}}}z^{2k\alpha}+O(z^{3k\alpha})}.
\]
Choose now $M=M_{p}=\diag(\mu_{1},\ldots,\mu_{d})$ such that $M_{p}^{p}=\Lambda$
and $\mu^{\alpha}=\zeta_{p}$, so $F_{1}$ is one-resonant of index
$p\alpha$. For $a=c/p$ and $b=-\frac{p-1}{2}a^{2}(k|\alpha|+1)$,
we then have
\[
F_{1}^{p}(z)=\Lambda z(1+cz^{k\alpha})+O(z^{3k\alpha})=F_{0}(z)+O(z^{3k\alpha}).
\]
Now, by the construction of normal forms for one-resonant germs in
\cite[Theorem~3.6]{BracciZaitsev2013Dynamicsofoneresonantbiholomorphisms},
for any $l\in\mathbb{N}$, there exists a local holomorphic change
of coordinates of the form $\chi(z)=z(1+O(\norm z^{3k|\alpha|}))$
such that
\[
\chi\circ F_{1}^{p}\circ\chi^{-1}(z)=F_{0}(z)+O(\norm z^{l}).
\]
The map $F_{1}$ under this change of coordinates becomes
\begin{align*}
F_{p}(z):=\chi\circ F_{1}\circ\chi^{-1}(z) & =M_{p}z\paren{1+az^{k\alpha}+bz^{2k\alpha}+O({{\norm z}}^{3k|\alpha|})}\\
 & =M_{p}z(1+az^{k\alpha})+O(\norm z^{2k|\alpha|+1}),
\end{align*}
and for any $F(z)=F_{p}(z)+O(\norm z^{l})$, we have
\[
F^{p}(z)=F_{p}^{p}(z)+O(\norm z^{l})=F_{0}(z)+O(\norm z^{l}).\qedhere
\]
\end{proof}
Applying Lemma~\ref{lem:RootGerm} to $F_{0}=F_{{\rm N}}$ as in
(\ref{eq:FNormalForm}) and $l>2kd+1$, shows that for every $p\in\mathbb{N}$
dividing $k$, there exists a germ $F_{p}$ of the form
\[
F_{p}(z)=M_{p}z\paren [\Big]{1-\frac{(z^{1}\cdots z^{d})^{k}}{kdp}}+O(\norm z^{2kd+1}),
\]
where $M_{p}=\diag(\mu_{1},\ldots,\mu_{d})$ with $\mu_{1}\cdots\mu_{d}=\zeta_{p}$,
such that whenever
\begin{equation}
G(z)=F_{p}(z)+O(\norm z^{l}),\label{eq:FRootWithTail}
\end{equation}
we have $G^{p}(z)=F_{{\rm N}}(z)+O(\norm z^{l})$. Again by Theorem~\ref{thm:JetInterp},
there exists an Automorphism $F\in\Aut(\mathbb{C}^{d})$ of the form
(\ref{eq:FRootWithTail}). In this case, $G^{p}(z)$ is an automorphism
of the form (\ref{eq:BiholWithTail}) and has $k$ invariant, non-recurrent,
attracting Fatou components $\Omega_{0},\ldots,\Omega_{k-1}$ at $0$
each biholomorphic to $\mathbb{C}\times(\mathbb{C}^{*})^{d-1}$ via
Proposition~\ref{prop:FatouBihol}, containing the corresponding
local basins $B_{0},\ldots,B_{k-1}$ from Theorem~\ref{thm:BZ}.
Hence, as in dimension $1$, for each $h\in\{0,\ldots,k-1\}$, $\Omega_{h}$
is part of a periodic cycle of Fatou components for $G$ whose period
divides $p$.

To show that the period is equal to $p$, note that for $r>0$ sufficiently
small $z_{r}:=(r,\ldots,\zeta_{k}^{h}r)\in B_{h}$ for each $h\in\{0,\ldots,k-1\}$.
Let $\pi(z)=z^{1}\cdots z^{d}$ for $z=(z^{1},\ldots,z^{d})\in B_{h}$
as usual and $\zeta_{m}:=e^{2\pi i/m}$ for $m\in\mathbb{N}$. Then
\begin{align*}
\pi(F(z_{r})) & =\zeta_{p}\zeta_{k}^{h}r^{d}\paren [\Big]{1-\frac{r^{kd}}{kdp}}^{d}+O(r^{2kd+d})\\
 & =\zeta_{k}^{h+k/p}r^{d}+O(r^{(k+1)d})
\end{align*}
and if $r$ is sufficiently small, we have $\pi(F(z_{r}))\in S_{h+k/p}(R_{0},\theta_{0})$,
and hence $F(z_{r})\in B_{h+k/p}$ (indices modulo $k$). This shows
that $F$ maps $B_{h}$ to $B_{h+k/p}$ and hence the period of $\Omega_{h}$
is equal to $p$, concluding the proof of Theorem~\ref{thm:FatouCCstarPeriodic}.

To derive Corollary~\ref{cor:CCstarVariants}, take an automorphism
$F$ of $\mathbb{C}^{m+1}$ with $k/p$ attracting cycles of period
$p$ from Theorem~\ref{thm:FatouCCstarPeriodic} and set
\[
G:\mathbb{C}^{d}\to\mathbb{C}^{d},\quad(z,w)\mapsto\paren [\Big]{F(z),\frac{1}{2}w}\quad\text{for }z\in\mathbb{C}^{m+1}\text{ and }w\in\mathbb{C}^{d-m-1}.
\]
Then the $w$ component of $\{G^{n}\}_{n\in\mathbb{N}}$ is locally
uniformly convergent to $0$ on all of $\mathbb{C}^{d-m-1}$, so any
subsequence $\{G^{n_{\ell}}\}_{\ell\in\mathbb{N}}$ converges locally
uniformly around $(z,w)\in\mathbb{C}^{d}$ if and only if $\{F^{n_{\ell}}\}_{\ell\in\mathbb{N}}$
does so around $z$. Thus $(z,w)$ is in the Fatou set of $G$ if
and only if $z$ is in the Fatou set of $F$ and the Fatou components
of $G$ are precisely of the form $U\times\mathbb{C}^{d-m-1}$ where
$U$ is a Fatou component of $F$. If $U$ is non-recurrent, $p$-periodic
and attracting to the origin, then so is $U\times\mathbb{C}^{d-m-1}$.

\end{document}